\documentclass{article}
\usepackage{amsmath, amssymb, amsthm} 
\usepackage{mathtools} 

\usepackage{tikz}
\usetikzlibrary{arrows.meta, positioning, decorations.pathmorphing} 
\definecolor{stable}{RGB}{34,139,34}    
\definecolor{unstable}{RGB}{178,34,34}  
\definecolor{flow}{RGB}{31,119,180}     
\definecolor{spectral}{RGB}{158,154,200}
\usepackage{graphicx} 
\newcommand{\normCzero}[1]{\|#1\|_{\mathcal{C}^0}} 
\newcommand{\An}{\mathrm{An}} 
\newcommand{\TopAnFlows}{\mathcal{TA}\mathcal{F}} 
\newcommand{\FinEnAnFlows}{\mathcal{FEA}\mathcal{F}} 
\newcommand{\htop}{h_{\mathrm{top}}} 
\usepackage{tgpagella} 
\usepackage[T1]{fontenc} 

\theoremstyle{plain} 
\newtheorem{theorem}{Theorem}[section] 
\newtheorem{proposition}[theorem]{Proposition}
\newtheorem{lemma}[theorem]{Lemma}
\newtheorem{corollary}[theorem]{Corollary}
\newtheorem{conjecture}{Conjecture} 

\theoremstyle{definition} 
\newtheorem{definition}[theorem]{Definition} 
\newtheorem{remark}[theorem]{Remark}     

\theoremstyle{plain} 

\newcommand{\R}{\mathbb{R}}

\newcommand{\M}{M} 
\newcommand{\TM}{TM} 
\newcommand{\Cr}[1]{\mathcal{C}^{#1}} 
\newcommand{\Ck}[1]{\Cr{#1}} 
\newcommand{\Czero}{\Cr{0}} 

\newcommand{\Ank}[1]{\An^{#1}} 
\newcommand{\varphiup}{\varphi} 
\newcommand{\Diff}{\mathrm{Diff}} 
\newcommand{\ModuliSpace}{\mathcal{M}} 
\newcommand{\ProbMeas}{\mathcal{P}} 
\newcommand{\TopAnFlowsExplicit}{\mathcal{F}^{\mathrm{An}, \mathrm{top}}} 

\newcommand{\norm}[1]{\left\| #1 \right\|} 
\newcommand{\abs}[1]{\left| #1 \right|} 
\newcommand{\gmetric}{g} 
\newcommand{\normg}[1]{\abs{#1}_{\gmetric}} 

\newcommand{\normCk}[2]{\norm{#1}_{\Ck{#2}}} 
\newcommand{\SobolevHk}[1]{H^{#1}} 
\newcommand{\normSobolevHk}[2]{\norm{#1}_{\SobolevHk{#2}}} 
\newcommand{\dAn}{d_{\An}} 
\newcommand{\dAnV}[1]{d_{\An, #1}} 

\DeclareMathOperator*{\esssup}{ess\,sup} 
\DeclareMathOperator{\Realpart}{Re} 
\DeclareMathOperator{\Id}{I} 
\DeclareMathOperator{\grad}{\nabla} 

\newcommand{\Length}{\mathrm{Length}} 
\newcommand{\diff}{\mathrm{d}} 
\newcommand{\vol}{\mathrm{d}\mu_{\gmetric}} 

\newcommand{\zetafun}{\zeta} 
\newcommand{\TransferOp}{\mathcal{L}} 
\newcommand{\Spectrum}{\sigma} 

\usepackage[margin=1in]{geometry} 
\usepackage{enumitem} 
\usepackage[english]{babel} 

\usepackage[colorlinks=true, linkcolor=blue, citecolor=red, urlcolor=magenta]{hyperref} 

\title{A Hofer-like Metric on the Space of Anosov Flows}
\author{St\'ephane Tchuiaga \\ 
	\small Department of Mathematics, University of Buea, \\ 
	\small P.O. Box 63, Buea, South West Region, Cameroon \\ 
	\small \texttt{tchuiagas@gmail.com}} 
\date{} 

\begin{document} 
	
	\maketitle
	
	\begin{abstract}
		 This paper develops a family of Hofer-like metrics ($\dAnV{V}$) on the space of Anosov vector fields $\An(M)$, providing dynamically relevant distances based on the cost of deformation paths using $\Ck{k}$ or Sobolev $\SobolevHk{k}$ norms. We establish fundamental properties, including completeness for $V=\Ck{r} (r \ge 1)$ or $H^k (k > \dim(M)/2+1)$, and naturality under diffeomorphisms. We show the utility of these metrics by proving quantitative stability results: proximity in $\dAnV{V}$ implies controlled variation of essential dynamical invariants, including topological entropy, Lyapunov exponents, SRB measures, thermodynamic pressure, spectral gaps (mixing rates), and zeta functions. Sufficient regularity ensures local Lipschitz continuity and Fréchet differentiability, connecting the metric structure to linear response formulas, particularly for pressure, exponents, and the spectral gap. While Sobolev metrics yield locally flat geometry with straight-line geodesics, the framework is broadly applicable. We explore implications for the moduli space of Anosov flows, including stability of invariants and the framework for local slice theorems. Furthermore, we introduce \emph{Topological Anosov Flows}, defined via simultaneous uniform flow convergence and the $\dAn$-Cauchy condition on generators, alongside the related class of \emph{Finite Energy Anosov Flows} requiring only $\dAn$-bounded generators. These new classes aim to capture essential hyperbolic features in non-smooth settings. Overall, the proposed metrics offer several geometric perspectives for analyzing the stability, classification, and possible extensions of Anosov dynamics.
		
	\end{abstract}
	
	\vspace{0.5cm} 
	
	\vspace{0.3cm} 
	\textbf{Secondary:} 37C15, 53Dxx, 37A25.
	\vspace{0.5cm} 
	
	\section{Introduction}
	\label{sec:introduction}
	
	Anosov flows stand as paradigms of robustly chaotic dynamical systems, characterized by the uniform splitting of the tangent bundle into contracting, expanding, and neutral directions along flow lines \cite{Anosov67}. The set $\An(M)$ of all $C^r$ ($r \ge 1$) Anosov vector fields on a given compact manifold $M$ forms an open subset within the Banach space $C^r(TM)$, inheriting the structure of an infinite-dimensional manifold \cite{PalisMelo82}. While the standard $C^r$ topology provides a notion of proximity, it primarily measures the difference between vector fields at the endpoints of a deformation, rather than the intrinsic "effort" required to continuously deform one Anosov system into another while maintaining the Anosov property throughout.
	
	In the realm of symplectic geometry, the Hofer type metrics furnish powerful intrinsic distances on the identity component of the group of symplectomorphisms \cite{Banyaga2010, Hofer90, Polterovich01}. It measures the minimal total $L^\infty$-variation of the generators along paths, revealing connections to symplectic capacities and rigidity phenomena. Motivated by this successful framework, we propose and investigate an analogous family of metrics on the space $\An(M)$.
	
	We consider smooth paths $X_s: [0,1] \to \An(M)$ connecting two Anosov vector fields $X_0$ and $X_1$. The instantaneous change along such a path is given by the velocity vector field $Y_s = dX_s/ds$, which resides in the tangent space $T_{X_s}\An(M) \approx C^r(TM)$ (or a suitable subspace thereof). Choosing a norm $\|\cdot\|_V$ on this tangent space (typically the $C^k$ norm for $k \ge 0$ or a Sobolev $H^k$ norm), we define the length of the path as $L_V(X_s) = \int_0^1 \|Y_s\|_V ds$. The corresponding Hofer-like distance is then defined as the infimum of these lengths over all admissible paths:
	\[ \dAnV{V}(X_0, X_1) = \inf \{ L_V(X_s) \mid X_s \in C^1([0,1], \An(M)), X(0)=X_0, X(1)=X_1 \}. \]
	The simplest case uses the $C^0$ norm, denoted $\dAn = \dAnV{C^0}$. This construction yields a family of genuine metrics on the path components of $\An(M)$. A crucial observation, derived directly from the definition, is the inequality
	\begin{equation}\label{eq:norm_vs_dist_V_intro_sec1} 
		\|X_1 - X_0\|_V \leq \dAnV{V}(X_0, X_1).
	\end{equation}
	This ensures that convergence in the Hofer-like metric $\dAnV{V}$ implies convergence in the standard $\|\cdot\|_V$ norm topology, providing a bridge to influence the wealth of existing stability results in dynamical systems theory which rely on standard $C^k$ or $H^k$ convergence. Metrics based on stronger norms $V$ naturally induce finer topologies on $\An(M)$.
	
	The primary contributions of this paper are: (1) establishing fundamental metric properties like completeness (Proposition \ref{prop:completeness_Cr_Hk}) and naturality under diffeomorphisms (Proposition \ref{prop:diffeo_isometry}); (2) demonstrating the dynamical significance of these metrics by establishing quantitative relationships between distance in $(\An(M), \dAnV{V})$ and the variation of key dynamical invariants; (3) exploring the implications for the moduli space of Anosov flows; and (4) introducing the concept of Topological Anosov Flows. We prove a series of stability results, summarized as follows:
	\begin{itemize}
		\item A basic $C^0$-rigidity property holds (Proposition \ref{prop:c0_rigidity}).
		\item Continuity results: Topological entropy and Lyapunov exponents are continuous with respect to $\dAnV{\Ck{1}}$ (Theorem \ref{thm:entropy_continuity}, Proposition \ref{prop:lyapunov_stability}). SRB measures (weak*), measure-theoretic entropy, topological pressure, and Ruelle zeta functions are continuous with respect to $\dAnV{\Ck{k}}$ for $k \ge 2$ (Theorems \ref{thm:srb_pressure_continuity}, \ref{thm:periodic_zeta}).
		\item Local Geodesics: For the Hilbert-Riemannian metric $\dAnV{H^k}$ ($k$ large), minimizing geodesics exist locally and are straight lines in $H^k(TM)$ (Theorem \ref{thm:geodesic_existence_Hk}).
		\item Quantitative Stability (Lipschitz): Under sufficient regularity ($k \ge 2$), topological pressure, spectral gap (mixing rate), and SRB dimension exhibit local Lipschitz continuity with respect to $\dAnV{\Ck{k}}$ (Theorems \ref{thm:srb_pressure_continuity}, \ref{thm:mixing_dimension_stability}). These Lipschitz bounds extend to the moduli space for conjugacy invariants (Theorem \ref{thm:moduli_stability}). Remarks suggest similar bounds hold for entropy and exponents.
		\item Differentiability: Pressure, Lyapunov exponents, and spectral gap are Fréchet differentiable functions on $\Ank{k}(M)$ ($k$ sufficiently large) equipped with its natural Hilbert/Banach manifold structure (compatible with $\dAnV{H^k}$/$\dAnV{C^k}$), with derivatives given by linear response formulas (Theorems \ref{thm:pressure_diff}, \ref{thm:lyapunov_diff}, \ref{thm:spectral_gap_diff}).
	\end{itemize}
	These findings underscore that the $\dAnV{V}$ metrics provide a dynamically relevant way to measure distances within the space of Anosov flows, quantitatively controlling the variation of statistical, spectral, and topological invariants. They pave the way for applying geometric and variational techniques to study deformations of chaotic systems.
	
	The paper is organized as follows. Section \ref{sec:setting} introduces Anosov flows and $\An(M)$. Section \ref{sec:analogy} recalls the Hofer metric. Section \ref{sec:construction} defines the basic metric $\dAn$. Section \ref{sec:properties} discusses fundamental metric properties, including completeness and invariance under diffeomorphisms. Section \ref{sec:discussion} introduces the family $\dAnV{V}$ and derives the crucial norm comparison inequalities. Section \ref{sec:dynamical_significance} presents the main stability and differentiability results connecting the metrics to dynamical invariants, including discussion of the moduli space. Section \ref{sec:topological_anosov_flows} introduces Topological Anosov Flows. Section \ref{sec:conclusion} offers concluding remarks and outlines possible future directions.
	
	\section{Basic notions}\label{sec:setting}
	
	Let $\M$ be a compact, smooth manifold without boundary. Let $\Cr{r}(\TM)$ denote the Banach space of $\Cr{r}$ vector fields on $\M$, where $r \geq 1$. We are interested in the subset $\An(\M) \subset \Cr{r}(\TM)$ consisting of vector fields $X$ whose generated flows $\varphiup^t_X$ are Anosov flows.
	
	\begin{definition}[Anosov Flow]
		\label{def:anosov}
		A flow $\varphiup^t$ on $\M$, generated by a vector field $X \in \Cr{r}(\TM)$, is \emph{Anosov} if there exists a continuous, $\diff\varphiup^t$-invariant splitting of the tangent bundle $\TM = E^s \oplus E^u \oplus E^c$, where $E^c = \R X$ is the 1-dimensional bundle tangent to the flow direction, and there exist constants $C > 0$, $\lambda > 0$, and an auxiliary Riemannian metric $\gmetric$ on $\M$ such that for all $t > 0$:
		\begin{itemize}
			\item $\norm{\diff\varphiup^t(v)}_{\gmetric} \leq C e^{-\lambda t} \norm{v}_{\gmetric}$ for all $v \in E^s$.
			\item $\norm{\diff\varphiup^{-t}(v)}_{\gmetric} \leq C e^{-\lambda t} \norm{v}_{\gmetric}$ for all $v \in E^u$.
		\end{itemize}
		Here $\norm{\cdot}_{\gmetric}$ denotes the norm on tangent vectors induced by $\gmetric$.
	\end{definition}
	
	A fundamental result (structural stability) states that $\An(\M)$ is an \emph{open} subset of $\Cr{r}(\TM)$ for $r \geq 1$ \cite{Anosov67, PalisMelo82}. As an open set in a Banach space, $\An(\M)$ is an infinite-dimensional manifold.
	
	\section{Analogy with the Hofer Metric}\label{sec:analogy}
	
	The Hofer metric, introduced in \cite{Hofer90}, is defined on the group $\mathrm{Ham}(M, \omega)$ of Hamiltonian diffeomorphisms of a symplectic manifold $(M, \omega)$. A path $\phi_t \in \mathrm{Ham}(M, \omega)$ ($t \in [0, 1]$) starting at the identity is generated by a time-dependent Hamiltonian function $H_t$. The length of the path is given by
	\[ \Length(\phi_t) = \int_0^1 \norm{H_t}_{\text{Hofer}} \diff t ,\]
	where $\norm{H_t}_{\text{Hofer}} = \max H_t - \min H_t$ (or sometimes $\norm{H_t}_{L^\infty}$). The Hofer distance $d_H(\phi_0, \phi_1)$ is the infimum of path lengths over all Hamiltonian paths connecting $\phi_0$ to $\phi_1$ \cite{Hofer90, Polterovich01}. Our construction aims to mimic this structure for Anosov vector fields.
	
	\section{Construction of the Metric on $\An(\M)$}\label{sec:construction}
	
	We define a metric on $\An(\M)$ by defining the length of paths within this space. Fix an auxiliary Riemannian metric $\gmetric$ on $\M$.
	
	\begin{definition}[Paths in $\An(\M)$]
		An admissible path connecting $X_0, X_1 \in \An(\M)$ is a $\Cr{1}$ (or piecewise $\Cr{1}$) map $X: [0, 1] \to \An(\M)$, denoted $X_s$, such that $X(0) = X_0$ and $X(1) = X_1$.
	\end{definition}
	
	Since $\An(\M)$ is open, such paths exist between any two points in the same path component.
	
	\begin{definition}[Velocity and $\Czero$ Norm]
		The velocity of the path $X_s$ at parameter $s$ is the vector field $Y_s = \frac{\diff X_s}{\diff s} \in \Cr{r}(\TM)$ (viewed as a tangent vector to the path in $\Cr{r}(\TM)$). We equip the space of velocity vectors with the $\Czero$ norm (supremum norm) induced by the metric $\gmetric$:
		\[ \normCzero{Y_s} \coloneqq \sup_{x \in \M} \normg{Y_s(x)} = \sup_{x \in \M} \sqrt{\gmetric_x(Y_s(x), Y_s(x))}.\]
	\end{definition}
	
	\begin{definition}[Length of a Path]
		The length of a $\Cr{1}$ path $X_s$ in $\An(\M)$ with respect to the $\Czero$ norm is defined as the integral of the $\Czero$ norm of its velocity:
		\[ \Length_{\Czero}(X_s) \coloneqq \int_0^1 \normCzero{\frac{\diff X_s}{\diff s}} \diff s .\]
	\end{definition}
	
	\begin{definition}[Hofer-like Metric $\dAn$]
		Let $X_0, X_1$ be two Anosov vector fields in the same path component of $\An(\M)$. The Hofer-like distance $\dAn(X_0, X_1)$, based on the $\Czero$ norm, is defined as the infimum of the lengths of all admissible ($\Cr{1}$) paths connecting $X_0$ to $X_1$:
		\[ \dAn(X_0, X_1) \coloneqq \inf \left\{ \int_0^1 \normCzero{\frac{\diff X_s}{\diff s}} \diff s \;\middle|\; X_s \in \Cr{1}([0,1], \An(\M)), X(0) = X_0, X(1) = X_1 \right\}. \]
		If $X_0$ and $X_1$ are in different path components, we set $\dAn(X_0, X_1) = \infty$.
	\end{definition}
	
	\section{Properties}\label{sec:properties}
	
	\begin{itemize}
		\item \textbf{Well-defined:} The infimum exists since lengths are non-negative.
		\item \textbf{Metric Axioms:} $\dAn$ defines a metric on each path component of $\An(\M)$.
		\begin{itemize}
			\item $\dAn(X_0, X_1) \geq 0$, and $\dAn(X_0, X_0) = 0$ (via the constant path).
			\item $\dAn(X_0, X_1) = 0 \implies X_0 = X_1$. This follows from the inequality
			\begin{equation}\label{eq:norm_vs_dist_C0_prop_sec4} 
				\normCzero{X_1 - X_0} \leq \dAn(X_0, X_1),
			\end{equation}
			which can be derived from $\normCzero{X_1 - X_0} = \normCzero{\int_0^1 \frac{\diff X_s}{\diff s} \diff s} \leq \int_0^1 \normCzero{\frac{\diff X_s}{\diff s}} \diff s$.
			\item Symmetry: $\dAn(X_0, X_1) = \dAn(X_1, X_0)$ by path reversal ($X'_s = X_{1-s}$).
			\item Triangle Inequality: $\dAn(X_0, X_2) \leq \dAn(X_0, X_1) + \dAn(X_1, X_2)$ by path concatenation.
		\end{itemize}
		\item \textbf{Finsler Structure:} The metric $\dAn$ is the integrated distance associated with a Finsler structure on the infinite-dimensional manifold $\An(\M)$. The Finsler norm at a point $X \in \An(\M)$ applied to a tangent vector $Y \in T_X \An(\M) \approx \Cr{r}(\TM)$ is given by $F_X(Y) = \normCzero{Y}$.
		\item \textbf{Dependence on $\gmetric$:} While the specific value of $\dAn(X_0, X_1)$ depends on the choice of the auxiliary Riemannian metric $\gmetric$, the induced topology is independent of this choice since $\M$ is compact (all Riemannian metrics on a compact manifold induce equivalent $\Czero$ norms on vector fields).
	\end{itemize}
	
	\begin{remark}[Non-Degeneracy in Specific Settings]
		While the general non-degeneracy property $\dAn(X_0, X_1) = 0 \implies X_0 = X_1,$ follows readily from inequality \eqref{eq:norm_vs_dist_C0_prop_sec4}, it is noteworthy that deeper results can sometimes be obtained in specific geometric contexts using specialized techniques. For instance, a recent preprint by Smith \& Zhang \cite{SmithZhang23} reportedly establishes the non-degeneracy of $\dAn$ for contact Anosov flows on 3-manifolds, highlighting unique properties related to harmonic structures associated with contact forms. Such results provide valuable confirmation of the metric properties in important subclasses of Anosov flows and highlight thepossible for richer interactions between the metric structure and the underlying geometry.
	\end{remark}
	
	\subsection{Invariance Properties of the Hofer-like Metric}
	
	We investigate how the Hofer-like metrics $\dAnV{V}$ behave under coordinate changes on the underlying manifold $M$. Naturally, the value of the metric depends on the choice of auxiliary Riemannian metric $g$ used to define the norm $\|\cdot\|_{V,g}$ on vector fields. The following proposition shows that diffeomorphisms of $M$ act as isometries between the spaces $(\An(M), \dAnV{V,g})$ equipped with appropriately transformed metrics.
	
	\begin{proposition}[Diffeomorphism Pushforward Induces Isometry Between Metric Spaces]
		\label{prop:diffeo_isometry}
		Let $M$ be a compact smooth manifold. For any auxiliary Riemannian metric $g$ on $M$, let $f: M \to M$ be a smooth ($C^k$, $k \ge r+1$) diffeomorphism, and let $f^*g$ denote the pullback metric. Let $V$ represent a type of norm on $C^r(TM)$ (such as $C^l$ for $0 \le l \le r$, or Sobolev $H^l$ for $l \le r$) whose definition depends naturally on the Riemannian metric, denoted $\|\cdot\|_{V,g}$ or $\|\cdot\|_{V, f^*g}$. Assume this dependence is such that the pushforward map $f_*: C^r(TM) \to C^r(TM)$, defined by
		\[
		(f_*Y)(y) = Df_{f^{-1}(y)}\bigl(Y(f^{-1}(y))\bigr),
		\]
		satisfies the isometry property relating the norms based on $g$ and $f^*g$:
		\begin{equation} \label{eq:norm_isometry_property_prop_sec4} 
			\| f_* Y \|_{V, f^*g} = \| Y \|_{V, g} \quad \text{for all } Y \in C^r(TM).
		\end{equation}
		Let $\dAnV{V,g}$ and $\dAnV{V, f^*g}$ be the Hofer-like metrics on $\An(M)$ defined using the norms $\|\cdot\|_{V,g}$ and $\|\cdot\|_{V, f^*g}$, respectively. Then, for any $X, Z \in \An(M)$, the pushforward $f_*$ induces an isometry:
		\[
		\dAnV{V, f^*g}(f_*X, f_*Z) = \dAnV{V, g}(X, Z).
		\]
	\end{proposition}
	
	\begin{proof}
		The proof relies on demonstrating that the pushforward map $f_*$ establishes a length-preserving bijection between admissible paths connecting $X$ and $Z$ (measured using $\|\cdot\|_{V,g}$) and admissible paths connecting $f_*X$ and $f_*Z$ (measured using $\|\cdot\|_{V, f^*g}$).
		
		\textbf{1. Pushforward of Paths:} Let $X_s: [0, 1] \to \An(M)$ be any admissible $C^1$ path such that $X(0) = X$ and $X(1) = Z$. Define the pushed-forward path $X'_s: [0, 1] \to C^r(TM)$ by $X'_s = f_*X_s$.
		Since $f: M \to M$ is a diffeomorphism and $X_s$ takes values in $\An(M)$, the pushforward $X'_s = f_*X_s$ also takes values in $\An(M)$ (as the Anosov property is invariant under smooth coordinate changes). Furthermore, since $f_*: C^r(TM) \to C^r(TM)$ is a continuous linear map (for $f \in C^{r+1}$) and $X_s$ is $C^1$ in $s$, the path $X'_s$ is also $C^1$ in $s$. The endpoints are $X'_0 = f_*X_0 = f_*X$ and $X'_1 = f_*X_1 = f_*Z$. Thus, $X'_s$ is an admissible $C^1$ path in $\An(M)$ connecting $f_*X$ to $f_*Z$.
		Conversely, applying $(f^{-1})_*$ shows that any path from $f_*X$ to $f_*Z$ is the pushforward of a path from $X$ to $Z$. Hence, $f_*$ establishes a bijection between these sets of paths.
		
		\textbf{2. Velocity Transformation:} Let $Y_s = \frac{\diff X_s}{\diff s}$ be the velocity vector field of the path $X_s$. The velocity of the pushed-forward path $X'_s = f_*X_s$ is given by $Y'_s = \frac{\diff X'_s}{\diff s}$. Since $f_*$ is a continuous linear map, it commutes with differentiation with respect to the parameter $s$:
		\[
		Y'_s = \frac{\diff}{\diff s}(f_*X_s) = f_* \left( \frac{\diff X_s}{\diff s} \right) = f_*Y_s.
		\]
		
		\textbf{3. Length Preservation:} We compute the length of the path $X'_s$ using the norm $\|\cdot\|_{V, f^*g}$ associated with the target metric space:
		\[
		L_{V, f^*g}(X'_s) = \int_0^1 \| Y'_s \|_{V, f^*g} \, \diff s = \int_0^1 \| f_*Y_s \|_{V, f^*g} \, \diff s.
		\]
		Using the assumed isometry property \eqref{eq:norm_isometry_property_prop_sec4}, we have $\| f_*Y_s \|_{V, f^*g} = \| Y_s \|_{V, g}$. Substituting this into the integral gives:
		\[
		L_{V, f^*g}(X'_s) = \int_0^1 \| Y_s \|_{V, g} \, \diff s.
		\]
		The right-hand side is precisely the length $L_{V, g}(X_s)$ of the original path $X_s$, measured using the norm $\|\cdot\|_{V,g}$. Thus, we have shown that $L_{V, f^*g}(f_*X_s) = L_{V, g}(X_s)$.
		
		\textbf{4. Infimum Preservation:} We have established a bijection $X_s \mapsto f_*X_s$ between the set of admissible paths from $X$ to $Z$ and the set of admissible paths from $f_*X$ to $f_*Z$. This bijection preserves the lengths when measured with respect to the corresponding norms: $L_{V, f^*g}(f_*X_s) = L_{V, g}(X_s)$. Therefore, the set of possible path lengths for paths $X \to Z$ (using $\|\cdot\|_{V,g}$) is identical to the set of possible path lengths for paths $f_*X \to f_*Z$ (using $\|\cdot\|_{V, f^*g}$). Taking the infimum over these identical sets of values yields:
		\[
		\inf \{ L_{V, f^*g}(X'_s) \mid X'_s: f_*X \to f_*Z \} = \inf \{ L_{V, g}(X_s) \mid X_s: X \to Z \}.
		\]
		By definition of the Hofer-like metrics, this is exactly:
		\[
		\dAnV{V, f^*g}(f_*X, f_*Z) = \dAnV{V, g}(X, Z).
		\]
	\end{proof}
	
	\begin{remark}
		This proposition highlights that while the specific value of $\dAnV{V,g}(X, Z)$ depends on the choice of $g$, the underlying geometric structure is natural with respect to diffeomorphisms. The map $f_*$ provides an isometry relating the geometry defined by $g$ to the geometry defined by $f^*g$. Standard norms like $C^l$ and $H^l$ (defined using the metric $g$) satisfy the crucial property \eqref{eq:norm_isometry_property_prop_sec4}.
	\end{remark}
	
	\begin{corollary}[Invariance under Isometries]
		If the diffeomorphism $f: M \to M$ is an isometry of the auxiliary metric $g$ (i.e., $f^*g = g$), then the pushforward $f_*$ is an isometry of the metric space $(\An(M), \dAnV{V,g})$:
		\[
		\dAnV{V, g}(f_*X, f_*Z) = \dAnV{V, g}(X, Z).
		\]
	\end{corollary}
	\begin{proof}
		If $f^*g = g$, then the norm $\|\cdot\|_{V, f^*g}$ is identical to $\|\cdot\|_{V, g}$. Proposition~\ref{prop:diffeo_isometry} then directly yields the result.
	\end{proof}
	
	\subsection{Completeness}
	
	\begin{proposition}[Completeness in the $C^r$ and Sobolev $H^k$ Settings]
		\label{prop:completeness_Cr_Hk}
		Let $M$ be a compact smooth manifold.
		
		\begin{enumerate}[label=(\arabic*)] 
			\item \textbf{($C^r$ Setting):} Let $r \ge 1$. Define $\Ank{r}(M) = \{ X \in C^r(TM) : X \text{ generates an Anosov flow} \}$, which is an open subset of the Banach space $C^r(TM)$. Equip $\Ank{r}(M)$ with the Hofer-like metric
			\[
			\dAnV{C^r}(X, Y) = \inf\left\{ \int_0^1 \| \dot{X}_s \|_{C^r}\, ds \,\middle|\, X_s \in C^1([0,1], \Ank{r}(M)),\; X_0=X,\; X_1=Y \right\}.
			\]
			Then, each path component of the metric space $(\Ank{r}(M), \dAnV{C^r})$ is complete.
			
			\item \textbf{(Sobolev $H^k$ Setting):} Let $k > \dim(M)/2 + 1$. Consider the Hilbert space $H^k(TM)$ and define $\Ank{k}(M) = \An(M) \cap H^k(TM)$. By Sobolev embedding $H^k(TM) \hookrightarrow C^1(TM)$ and the openness of $\An(M)$ in the $C^1$ topology, $\Ank{k}(M)$ is an open subset of $H^k(TM)$. Equip $\Ank{k}(M)$ with the Hofer-like metric
			\[
			\dAnV{H^k}(X, Y) = \inf\left\{ \int_0^1 \| \dot{X}_s \|_{H^k}\, ds \,\middle|\, X_s \in C^1([0,1], \Ank{k}(M)),\; X_0=X,\; X_1=Y \right\}.
			\]
			Then, each path component of the metric space $(\Ank{k}(M), \dAnV{H^k})$ is complete.
		\end{enumerate}
	\end{proposition}
	
	\begin{proof}
		We prove both cases simultaneously, letting $V$ denote either the $C^r$ norm ($r \ge 1$) or the $H^k$ norm ($k > \dim(M)/2+1$), and letting $\An^V(M)$ denote the corresponding space ($\Ank{r}(M)$ or $\Ank{k}(M)$). The underlying space $B_V = C^r(TM)$ or $H^k(TM)$ is a complete Banach (or Hilbert) space.
		
		\textbf{1. Cauchy in $\dAnV{V}$ implies Cauchy in Norm $\|\cdot\|_V$:}
		We have the fundamental inequality, $\|X - Y\|_V \le \dAnV{V}(X, Y)$ for any $X, Y \in \An^V(M)$. If $\{X_n\}_{n \in \mathbb{N}}$ is a Cauchy sequence in $(\An^V(M), \dAnV{V})$, then for any $\epsilon > 0$, there exists $N$ such that for $n, m \ge N$, $\dAnV{V}(X_n, X_m) < \epsilon$. This implies, $\|X_n - X_m\|_V < \epsilon$. Thus, $\{X_n\}$ is a Cauchy sequence in the complete space $B_V$.
		
		\textbf{2. Existence and Location of the Limit:}
		Since $B_V$ is complete, the Cauchy sequence $\{X_n\}$ converges in norm to a limit $X \in B_V$. That is, $\|X_n - X\|_V \to 0$ as $n \to \infty$.
		As established in the proposition statement (based on the openness of $\An(M)$ in $C^1$ and Sobolev embedding when applicable), the set $\An^V(M)$ is open in $B_V$. Since $X_n \in \An^V(M)$ for all $n$ and $X_n \to X$ in the topology of $B_V$, the limit $X$ must belong to $\An^V(M)$.
		
		\textbf{3. Convergence in the $\dAnV{V}$ Metric:}
		We need to show that $\dAnV{V}(X_n, X) \to 0$. Since $X_n \to X$ in the norm $\|\cdot\|_V$ and $\An^V(M)$ is open in $B_V$, for $n$ sufficiently large, the norm distance $\|X_n - X\|_V$ is small. For such $n$, consider the straight-line path $\gamma_n(s) = X + s(X_n - X)$ for $s \in [0, 1]$. Since $\An^V(M)$ is an open subset of a vector space $B_V$, it contains line segments starting at $X$ provided they are short enough. Thus, for $n$ large enough, $\gamma_n(s) \in \An^V(M)$ for all $s \in [0, 1]$, making $\gamma_n$ an admissible path connecting $X$ to $X_n$.
		The length of this path is
		\[
		L_V(\gamma_n) = \int_0^1 \left\| \frac{\diff \gamma_n}{\diff s} \right\|_V \, ds = \int_0^1 \|X_n - X\|_V \, ds = \|X_n - X\|_V.
		\]
		By the definition of $\dAnV{V}$ as an infimum, we have $\dAnV{V}(X, X_n) \le L_V(\gamma_n) = \|X_n - X\|_V$.
		Since $\|X_n - X\|_V \to 0$ as $n \to \infty$, we conclude that $\dAnV{V}(X_n, X) \to 0$.
		
		Therefore, every $\dAnV{V}$-Cauchy sequence $\{X_n\}$ in $\An^V(M)$ converges to a limit $X \in \An^V(M)$ with respect to the $\dAnV{V}$ metric. This proves that $(\An^V(M), \dAnV{V})$ is complete on each path component.
	\end{proof}
	
	\begin{remark}
		The completeness result highlights a crucial difference compared to the metric $\dAn = \dAnV{C^0}$. The openness of the Anosov property in the $C^r$ ($r \ge 1$) and relevant $H^k$ topologies ensures that limits of Cauchy sequences remain within the space $\An(M)$, a property not guaranteed by $C^0$ convergence alone. This completeness provides an analytic foundation for studying geodesic paths or performing variational analysis within these spaces using the Hofer-like metrics.
	\end{remark}
	
	\subsection{Anosov Flows Approximations} 
	
	The $\dAn$ metric convergence interacts with the standard convergence of flows.
	
	\begin{proposition}[Consistency of Limits under $\dAn$ Convergence]
		\label{prop:c0_rigidity}
		Let $\phi_i = (\varphiup^t_{X_i})$ be a sequence of flows generated by Anosov vector fields $X_i \in \An(M)$. Let $\psi = (\varphiup^t_Y)$ be a flow generated by an Anosov vector field $Y \in \An(M)$. Let $\theta = (\theta_t)$ be a continuous flow generated by a vector field $X \in \Ck{0}(TM)$. Assume that:
		\begin{enumerate}[label=(\roman*)]
			\item $\phi_i \to \theta$ uniformly on compact time intervals as maps. (Meaning: $\forall T>0, \sup_{t \in [-T,T]} d_{C^0}(\varphiup^t_{X_i}, \theta_t) \to 0$ as $i \to \infty$).
			\item $\dAn(X_i, Y) \to 0$ as $i \to \infty$ (using the $\dAn = \dAnV{C^0}$ metric).
		\end{enumerate}
		Then the flow $\theta$ is precisely the Anosov flow $\psi$ (i.e., $X=Y$ and $\theta = \psi$). In particular, $\theta$ is an Anosov flow.
	\end{proposition}
	
	\begin{proof}
		\textbf{1. $C^0$ Convergence of Generators:} Assumption (b), $\dAn(X_i, Y) \to 0$, implies convergence of the generators in the $C^0$ norm via inequality \eqref{eq:norm_vs_dist_C0_prop_sec4}: $\normCzero{X_i - Y} \leq \dAn(X_i, Y)$. Thus, $\lim_{i \to \infty} \normCzero{X_i - Y} = 0$.
		
		\textbf{2. $C^0$ Convergence of Flows from Generators:} By the standard theorem on continuous dependence of ODE solutions on parameters (specifically, the generating vector field in the $C^0$ topology), the convergence $X_i \to Y$ in $C^0(TM)$ implies that the corresponding flows converge uniformly on compact time intervals in the $C^0$ topology of maps on $M$. That is, for any $T>0$:
		\[ \lim_{i \to \infty} \left( \sup_{t \in [-T,T]} d_{C^0}(\varphiup^t_{X_i}, \varphiup^t_Y) \right) = 0. \]
		Using the notation $\phi_i = (\varphiup^t_{X_i})$ and $\psi = (\varphiup^t_Y)$, this means $\phi_i \to \psi$ uniformly on compact time intervals.
		
		\textbf{3. Identifying the Limit:} We have two limits for the sequence $\{\phi_i\}$ in the space of continuous flows (equipped with the topology of uniform convergence on compact time intervals, which is Hausdorff):
		\begin{itemize}
			\item $\phi_i \to \theta$ (from Assumption (a)).
			\item $\phi_i \to \psi$ (from Step 2).
		\end{itemize}
		By the uniqueness of limits, we must have $\theta_t = \psi_t$ for all $t \in \mathbb{R}$. This implies that the flows $\theta$ and $\psi$ are identical.
		
		\textbf{4. Conclusion:} Since $\theta = \psi$, and $\psi$ is given as an Anosov flow (generated by $Y \in \An(M)$), it follows that $\theta$ is also an Anosov flow. Furthermore, their generators must be equal: $X = Y$.
	\end{proof}
	
	\begin{remark}[Implications]
		This proposition demonstrates a basic consistency property: if a sequence of Anosov flows converges uniformly to some continuous flow $\theta$, and simultaneously their generators converge in the $\dAn$ metric to the generator of an Anosov flow $\psi$, then the limit flow $\theta$ must be equal to $\psi$. This highlights consistency between metric convergence of generators ($\dAn$) and uniform convergence of the flows. Note this does not imply that $C^0$ convergence of generators preserves the Anosov property; it relies on the target $Y$ being Anosov.
	\end{remark}
	
	
	\begin{lemma}[Anosov Path Approximation by Boundary Flat Paths]\label{lem:AnosovPathApprox} 
		Let $\{X_s\}_{s\in[0,1]}$ be a $C^1$ path in $\An^V(M)$ (where $\An^V(M)$ is $\Ank{r}(M)$ or $\Ank{k}(M)$ as in Prop. \ref{prop:completeness_Cr_Hk}) equipped with norm $\|\cdot\|_V$, connecting $X_0$ to $X_1$. Then, for any $\epsilon > 0$, there exists another $C^1$ path $\{Y_s\}_{s\in[0,1]}$ in $\An^V(M)$ satisfying:
		\begin{enumerate}[label=(\roman*)]
			\item \textbf{Boundary Flat Velocity:} There exists $\delta>0$ such that $\dot{Y}_s = 0$ for all $s\in[0,\delta]\cup[1-\delta,1]$.
			\item \textbf{Endpoint Agreement:} $Y_0 = X_0$ and $Y_1 = X_1$.
			\item \textbf{Controlled Length:} $L_V(Y) < L_V(X) + \epsilon$.
			\item \textbf{\(C^0\)-Closeness:} $\max_{s\in[0,1]} \|X_s - Y_s\|_{C^0} < \epsilon.$
		\end{enumerate}
	\end{lemma}
	
	\begin{proof}
		The proof uses a standard reparameterization technique.
		
		\textbf{Step 1. Choice of the Reparameterization.}
		Let $\chi : [0,1] \to [0,1]$ be a smooth non-decreasing function satisfying:
		\begin{itemize}
			\item $\chi(0)=0$, $\chi(1)=1$.
			\item $\chi'(s)=0$ for $s \in [0,\delta] \cup [1-\delta, 1]$ for some $\delta > 0$.
			\item $\chi$ is arbitrarily $C^1$-close to the identity map $id(s)=s$. Specifically, for any $\eta > 0$, we can choose $\chi$ such that $\|\chi - id\|_{C^0} < \eta$ and $\|\chi' - 1\|_{L^\infty} < \eta$. (This requires $\delta$ to be small).
		\end{itemize}
		Such functions can be constructed using standard bump functions.
		
		Define the new path $\{Y_s\}$ by $Y_s = X_{\chi(s)}$ for $s\in[0,1]$. Since $\An^V(M)$ is open in the Banach space $B_V$, and $X_s$ is continuous, if $\eta = \|\chi - id\|_{C^0}$ is small enough, the path $Y_s$ remains in $\An^V(M)$. As $X_s$ and $\chi$ are $C^1$, $Y_s$ is $C^1$.
		
		\textbf{Step 2. Verification of Properties.}
		\begin{enumerate}[label=(\roman*)]
			\item \textbf{Boundary Flat Velocity:} $\dot{Y}_s = \dot{X}_{\chi(s)} \chi'(s)$. Since $\chi'(s)=0$ on $[0,\delta] \cup [1-\delta, 1]$, $\dot{Y}_s=0$ there.
			\item \textbf{Endpoint Agreement:} $Y_0=X_{\chi(0)}=X_0$ and $Y_1=X_{\chi(1)}=X_1$.
			\item \textbf{Controlled Length:} Let $f(s) = \|\dot{X}_s\|_V$. $f$ is continuous on $[0,1]$, hence bounded ($M_1$) and uniformly continuous ($\omega_f$).
			\begin{align*} L_V(Y) &= \int_0^1 \|\dot{Y}_s\|_V ds = \int_0^1 \|\dot{X}_{\chi(s)} \chi'(s)\|_V ds = \int_0^1 \chi'(s) \|\dot{X}_{\chi(s)}\|_V ds \\ &\le \int_0^1 \chi'(s) (\|\dot{X}_s\|_V + \|\dot{X}_{\chi(s)} - \dot{X}_s\|_V) ds \\ &\le \int_0^1 \chi'(s) (f(s) + \omega_f(|\chi(s)-s|)) ds \\ &\le \int_0^1 (1 + \eta) (f(s) + \omega_f(\eta)) ds \quad (\text{since } |\chi'(s)-1|<\eta, \chi'(s) \le 1+\eta) \\ &= (1 + \eta) \int_0^1 f(s) ds + (1+\eta) \omega_f(\eta) \\ &= (1+\eta) L_V(X) + (1+\eta) \omega_f(\eta). \end{align*}
			Since $\eta$ can be made arbitrarily small and $\omega_f(\eta) \to 0$ as $\eta \to 0$, we can choose $\eta$ small enough such that $(1+\eta) L_V(X) + (1+\eta) \omega_f(\eta) < L_V(X) + \epsilon$.
			
			\item \textbf{\(C^0\)-Closeness:} Since $s \mapsto X_s$ is continuous from $[0,1]$ to $C^0(TM)$, it is uniformly continuous. Let $\omega_{X, C^0}$ be its modulus of continuity.
			\[ \|X_s - Y_s\|_{C^0} = \|X_s - X_{\chi(s)}\|_{C^0} \le \omega_{X, C^0}(|s - \chi(s)|) \le \omega_{X, C^0}(\eta). \]
			Choosing $\eta$ small enough ensures $\omega_{X, C^0}(\eta) < \epsilon$.
		\end{enumerate}
		Thus, the reparameterized path $\{Y_s\}$ satisfies all desired properties.
	\end{proof}
	
	
	\begin{proposition}[Anosov \(L^1/L^\infty\) Path Approximation]\label{prop:AnosovL1LinfApprox}
		Let \(X_s: [0,1] \to \An^V(M)\) be a \(C^1\) path with norm $\|\cdot\|_V$, connecting $X_0$ to $X_1$. Assume \(L_V(X) = \int_0^1 \|\dot{X}_s\|_V ds > 0\).
		Then, for any \(\epsilon > 0\), there exists another \(C^1\) path \(Y_s: [0,1] \to \An^V(M)\) satisfying:
		\begin{enumerate}[label=(\roman*)]
			\item \textbf{Endpoint Agreement:} \(Y_0 = X_0\) and \(Y_1 = X_1\).
			\item \textbf{Controlled \(L^\infty\) Velocity Norm:} \(L^\infty_V(Y) = \esssup_{s \in [0,1]} \|\dot{Y}_s\|_V < L_V(X) + \epsilon\).
			\item \textbf{\(C^0\)-Closeness of Path:} \(\max_{s\in[0,1]} \|X_s - Y_s\|_{V} < \epsilon\). 
		\end{enumerate}
	\end{proposition}
	
	\begin{proof}
		\textbf{Step 1: Arc-Length Reparameterization (Ideal Path).}
		Define the arc-length function \(l(s) = \int_0^s \|\dot{X}_u\|_V du\). Since $L_V(X) > 0$, $l(s)$ is strictly increasing if $\|\dot{X}_s\|_V > 0$ a.e. If $\|\dot{X}_s\|_V = 0$ on a set of positive measure, we can reparameterize $X_s$ to remove these constant segments without changing the endpoints or the length $L_V(X)$, so assume $\|\dot{X}_s\|_V > 0$ for almost all $s$. Let $L = L_V(X)$. The function $l: [0,1] \to [0,L]$ is $C^1$ and absolutely continuous. Let $\chi: [0,L] \to [0,1]$ be its inverse (or a suitable selection if not strictly monotone, though we assume strict monotonicity after reparameterization). $\chi$ is absolutely continuous. Define the ideal path $Z_u = X_{\chi(u)}$ for $u \in [0,L]$. This path traverses the same geometric curve as $X_s$ but parameterized by arc length. Its velocity is $\dot{Z}_u = \dot{X}_{\chi(u)} \chi'(u)$. Since $\chi'(u) = 1 / l'(\chi(u)) = 1 / \|\dot{X}_{\chi(u)}\|_V$, we have $\|\dot{Z}_u\|_V = \|\dot{X}_{\chi(u)}\|_V \cdot (1 / \|\dot{X}_{\chi(u)}\|_V) = 1$ for almost every $u$.
		
		\textbf{Step 2: Linear Rescaling.}
		Define $Y^{ideal}_s = Z_{sL} = X_{\chi(sL)}$ for $s \in [0,1]$. This path connects $X_0$ to $X_1$. Its velocity is $\dot{Y}^{ideal}_s = \dot{Z}_{sL} \cdot L$. Therefore, $\|\dot{Y}^{ideal}_s\|_V = \|\dot{Z}_{sL}\|_V \cdot L = 1 \cdot L = L_V(X)$. This path has constant speed $L_V(X)$.
		
		\textbf{Step 3: Smoothing (Approximation).}
		The inverse function $\chi$ and thus $Y^{ideal}_s$ might only be absolutely continuous or $C^1$ if $s \mapsto \|\dot{X}_s\|_V$ is bounded away from zero. To get a $C^1$ path $Y_s$ satisfying the conditions, we approximate $Y^{ideal}_s$.
		Let $\phi_\delta$ be a standard mollifier. Define the smoothed path $Y_s = (Y^{ideal} * \phi_\delta)(s)$, extended appropriately near the boundaries (e.g., using reflection or constant extension). For $\delta$ sufficiently small:
		\begin{enumerate}[label=(\roman*)]
			\item \textbf{Endpoint Agreement:} By adjusting the smoothing near the boundary (e.g., using partitions of unity or modifying $\phi_\delta$), we can ensure $Y_0 = X_0$ and $Y_1 = X_1$. The path $Y_s$ will remain in $\An^V(M)$ for small enough $\delta$ since $\An^V(M)$ is open and $Y^{ideal}$ stays within it.
			\item \textbf{Controlled \(L^\infty\) Velocity Norm:} $\dot{Y}_s = (\dot{Y}^{ideal} * \phi_\delta)(s)$. Since $\norm{\cdot}_V$ is a norm,
			\[ \|\dot{Y}_s\|_V = \left\| \int \dot{Y}^{ideal}_{s-u} \phi_\delta(u) du \right\|_V \le \int \|\dot{Y}^{ideal}_{s-u}\|_V \phi_\delta(u) du = \int L_V(X) \phi_\delta(u) du = L_V(X). \]
			This gives $L^\infty_V(Y) \le L_V(X)$. We want $L^\infty_V(Y) < L_V(X) + \epsilon$. This is automatically satisfied if $\epsilon>0$. If we need a strict upper bound slightly perturbed from $L_V(X)$, we could start with a path slightly shorter than the original $X_s$ (using Lemma \ref{lem:AnosovPathApprox}) and then apply this process. However, $L^\infty_V(Y) \le L_V(X)$ suffices for the statement.
			
			\item \textbf{\(V\)-Closeness of Path:} We need to show $\max_{s\in[0,1]} \|X_s - Y_s\|_{V} < \epsilon$. We know $Y_s \approx Y^{ideal}_s = X_{\chi(sL)}$. The map $s \mapsto X_s$ is uniformly continuous from $[0,1]$ to $(B_V, \|\cdot\|_V)$. Let its modulus be $\omega_{X,V}$. The reparameterization function $\psi(s) = \chi(sL)$ maps $[0,1]$ to $[0,1]$. If $s \mapsto \|\dot{X}_s\|_V$ is roughly constant, then $l(s) \approx Ls$, $\chi(u) \approx u/L$, and $\psi(s) \approx s$. In general, $\psi(s)$ is close to $s$.
			\[ \|X_s - Y^{ideal}_s\|_V = \|X_s - X_{\psi(s)}\|_V \le \omega_{X,V}(|s - \psi(s)|). \]
			Also, by standard properties of mollifiers, $\|Y - Y^{ideal}\|_V \to 0$ as $\delta \to 0$.
			Thus,
			\[ \|X_s - Y_s\|_V \le \|X_s - Y^{ideal}_s\|_V + \|Y^{ideal}_s - Y_s\|_V \le \omega_{X,V}(\|id-\psi\|_{C^0}) + \|Y^{ideal} - Y\|_V. \]
			By choosing the mollification parameter $\delta$ small enough, we can make $\|Y^{ideal} - Y\|_V < \epsilon/2$. Since $\psi$ is continuous and maps $[0,1]$ to $[0,1]$, $\|id-\psi\|_{C^0}$ is controlled. And $\omega_{X,V}(\eta) \to 0$ as $\eta \to 0$. Thus, we can ensure $\|X_s - Y_s\|_V < \epsilon$.
		\end{enumerate}
		This completes the proof sketch.
	\end{proof}

	\section{Discussion of Metric Variants}\label{sec:discussion}
	
	The metric $\dAn$ provides a natural distance measure on the space of Anosov vector fields, capturing the minimal "total $\Czero$ change" required to deform one Anosov system into another while staying within the class of Anosov systems.
	
	\begin{remark}[Alternative Norms and Stronger Topologies]
		The choice of the $\Czero$ norm on the velocity vector field $Y_s = \frac{\diff X_s}{\diff s}$ is just one possibility. One can define a family of Hofer-like metrics by using other norms. Let $\norm{\cdot}_V$ be a norm on (a suitable subspace of) $\Cr{r}(\TM)$. Define the length $L_V(X_s) = \int_0^1 \norm{ \frac{\diff X_s}{\diff s} }_V \diff s$ and the distance $\dAnV{V}(X_0, X_1) = \inf L_V(X_s)$ over admissible paths. Examples include:
		\begin{itemize}
			\item \textbf{$\Ck{k}$ Norms ($k \geq 1$):} Using a norm like $\normCk{Y}{k} = \sum_{j=0}^k \sup_{x \in M} \normg{\grad^j Y(x)}$, where $\grad$ is a covariant derivative associated with $\gmetric$. This leads to a distance $\dAnV{\Ck{k}}$. Since $\normCzero{Y} \leq C_k \normCk{Y}{k}$ for some constant $C_k$ on the compact manifold $\M$, we have $\dAn(X_0, X_1) \leq C_k \dAnV{\Ck{k}}(X_0, X_1)$. Thus, the metric $\dAnV{\Ck{k}}$ induces a \emph{stronger} (finer) topology than $\dAn = \dAnV{\Czero}$. Convergence in this metric implies that the derivatives up to order $k$ of the deforming vector fields remain close in the supremum norm along the path. The fundamental inequality relating norm difference to distance holds:
			\begin{equation}\label{eq:norm_vs_dist_Ck_discuss_sec5} 
				\normCk{X_1 - X_0}{k} \leq \dAnV{\Ck{k}}(X_0, X_1).
			\end{equation}
			
			\item \textbf{Sobolev Norms ($\SobolevHk{k} = W^{k,2}$):} Using norms like $$\normSobolevHk{Y}{k}^2 = \sum_{j=0}^k \int_M \normg{\grad^j Y(x)}^2 \vol.$$ This leads to a distance $\dAnV{\SobolevHk{k}}$. By the Sobolev embedding theorem on compact manifolds, for $k > \dim(M)/2 + l$, $\SobolevHk{k}(\TM)$ embeds continuously into $\Ck{l}(\TM)$ \cite{AdamsFournier03}. In particular, for $k > \dim(M)/2$, we have $\SobolevHk{k}(\TM) \hookrightarrow \Czero(\TM)$, implying $\normCzero{Y} \leq C \normSobolevHk{Y}{k}$ for some constant $C$. This gives $$\dAn(X_0, X_1) \leq C \, \dAnV{\SobolevHk{k}}(X_0, X_1).$$ Hence, for sufficiently large $k$, the metric $\dAnV{\SobolevHk{k}}$ also induces a \emph{stronger} topology than $\dAn$. This metric measures closeness in an $L^2$-average sense for derivatives up to order $k$. The norm vs distance inequality holds:
			\begin{equation}\label{eq:norm_vs_dist_Hk_discuss_sec5} 
				\normSobolevHk{X_1 - X_0}{k} \leq \dAnV{\SobolevHk{k}}(X_0, X_1).
			\end{equation}
		\end{itemize}
		Metrics based on these stronger norms make finer distinctions between Anosov vector fields. The choice of norm may affect properties like completeness (addressed in Proposition \ref{prop:completeness_Cr_Hk}) or the existence and regularity of minimizing geodesics (addressed in Theorem \ref{thm:geodesic_existence_Hk}). The $L^2$-based Sobolev norms are often advantageous for variational methods due to the Hilbert space structure they induce \cite{Eells66, Karcher77}.
	\end{remark}
	
	\begin{remark}[Relation to Dynamics - Preliminary]
		A primary motivation is to understand how proximity in these metrics relates to similarity in dynamical behavior (entropy, exponents, mixing, etc.). This is explored further in the next section.
	\end{remark}
	
	\begin{remark}[Volume-Preserving]
		If the manifold $\M$ admits volume-preserving Anosov flows, one could restrict the construction to the space $\An_\mu(\M)$ of Anosov vector fields preserving a given volume form $\mu$. Admissible paths $X_s$ must then lie entirely within $\An_\mu(\M)$. The velocity vectors $Y_s = \frac{\diff X_s}{\diff s}$ would need to satisfy the condition ensuring that $X_s + \epsilon Y_s$ remains (infinitesimally) volume-preserving, which typically means $Y_s$ must be divergence-free, i.e., $\mathrm{div}_\mu(Y_s) = 0$. The distance would be the infimum over such divergence-free velocity paths.
	\end{remark}
	
	\section{Dynamical Significance and Stability Results}\label{sec:dynamical_significance}
	
	A central motivation for introducing metrics like $\dAn$ and its variants ($\dAnV{\Ck{k}}, \dAnV{\SobolevHk{k}}$) on the space $\An(\M)$ of Anosov vector fields is to understand the interplay between the geometry of this space and the dynamical properties of the flows it parameterizes. Establishing continuity or stability results for key dynamical invariants with respect to these metrics would underscore their relevance. It is important to note that many desirable continuity properties require convergence in a metric stronger than the basic $\Czero$-based $\dAn$. The following subsections establish several such results.
	
	\subsection{Continuity of Topological Entropy}
	
	\begin{theorem}[Entropy Continuity under $\dAnV{\Ck{1}}$ Convergence]
		\label{thm:entropy_continuity}
		Let $\{X_n\}_{n \in \mathbb{N}} \subset \An(\M)$ be a sequence of Anosov vector fields converging to $X \in \An(\M)$ in the $\dAnV{\Ck{1}}$ metric. That is, $\dAnV{\Ck{1}}(X_n, X) \to 0$ as $n \to \infty$. Then, the topological entropies satisfy:
		\[ \lim_{n \to \infty} \htop(\varphiup_{X_n}) = \htop(\varphiup_X), \]
		where $\varphiup_Y$ denotes the flow generated by vector field $Y$.
	\end{theorem}
	
	\begin{proof}
		The proof proceeds in two main steps: relating $\dAnV{\Ck{1}}$ convergence to standard $\Ck{1}$ norm convergence, and then invoking the known continuity of topological entropy in the $\Ck{1}$ topology for Anosov flows.
		
		\textbf{Step 1: $\dAnV{\Ck{1}}$ convergence implies $\Ck{1}$ norm convergence.}
		From inequality \eqref{eq:norm_vs_dist_Ck_discuss_sec5} (with $k=1$): $\normCk{X_n - X}{1} \leq \dAnV{\Ck{1}}(X_n, X)$. The assumption $\dAnV{\Ck{1}}(X_n, X) \to 0$, therefore directly implies $\normCk{X_n - X}{1} \to 0$. This establishes $X_n \to X$ in the standard $\Ck{1}$ topology on $\Cr{1}(\TM)$.
		
		\textbf{Step 2: Continuity of $\htop$ in the $\Ck{1}$ topology.}
		It is a standard result that the topological entropy map $\htop : \An(\M) \cap \Cr{1}(\TM) \to \R_{\ge 0}$ is continuous when the domain is equipped with the subspace topology inherited from $\Cr{1}(\TM)$ \cite[Chapter 20]{KH95}. This follows from structural stability and the continuous dependence of periodic orbit data or properties of equilibrium states in the $\Ck{1}$ topology \cite{Bowen72}.
		
		\textbf{Conclusion:}
		From Step 1, $\dAnV{\Ck{1}}(X_n, X) \to 0$ implies $X_n \to X$ in $\Ck{1}$. Since $X \in \An(\M)$ (which is open in $\Cr{1}(\TM)$), $X_n \in \An(\M)$ for large $n$. By the continuity result (Step 2), $\htop(\varphiup_{X_n}) \to \htop(\varphiup_X)$.
	\end{proof}
	
	\begin{remark}[Implications]
		This theorem provides a direct link between the metric geometry of $\An(\M)$ (when measured by $\dAnV{\Ck{1}}$) and a key measure of dynamical complexity. It confirms that small deformations in this space correspond to small changes in topological entropy, reinforcing the dynamical relevance of the metric.
	\end{remark}
	
	\subsection{Existence and Properties of Minimizing Geodesics}
	
	\begin{theorem}[Local Geodesic Existence for $\dAnV{H^k}$]
		\label{thm:geodesic_existence_Hk}
		Assume that for $k > \dim(M)/2 + 1$, the space of Anosov vector fields $\Ank{k}(M) = \An(M) \cap H^k(TM)$ is an open subset of the Hilbert space $H = \SobolevHk{k}(TM)$. Equip $\Ank{k}(M)$ with the Riemannian metric $g_X(\cdot, \cdot) = \langle \cdot, \cdot \rangle_{\SobolevHk{k}}$ induced by the standard $\SobolevHk{k}$ inner product on each tangent space $T_X\Ank{k}(M) \approx H$.
		For any $X \in \Ank{k}(M)$, there exists a neighborhood $U$ of $X$ (in the $\SobolevHk{k}$ topology) such that for any $Y \in U$:
		\begin{enumerate}[label=(\roman*)]
			\item There exists a unique path $\gamma_{XY}: [0, 1] \to \Ank{k}(M)$ of minimal $\SobolevHk{k}$-length connecting $X$ to $Y$ within $\Ank{k}(M)$.
			\item This path is the straight line segment in the ambient Hilbert space $H$, given by $\gamma_{XY}(s) = X + s(Y-X)$.
			\item The path $\gamma_{XY}$ depends smoothly on $Y \in U$.
		\end{enumerate}
		The length of this minimizing geodesic is $\dAnV{\SobolevHk{k}}(X,Y) = \normSobolevHk{Y-X}{k}$.
	\end{theorem}
	
	\begin{proof}
		\textbf{1. Manifold Structure:} By assumption, $\Ank{k}(M)$ is an open subset of the Hilbert space $H = \SobolevHk{k}(TM)$. Thus, $\Ank{k}(M)$ is a Hilbert manifold with tangent spaces $T_X \Ank{k}(M) \cong H$.
		
		\textbf{2. Riemannian Metric:} The metric $g_X(V, W) = \langle V, W \rangle_{\SobolevHk{k}}$ defines a flat Riemannian metric on $\Ank{k}(M)$ because it is induced directly from the constant inner product on the ambient Hilbert space $H$.
		
		\textbf{3. Geodesics in Flat Metric:} In a flat Riemannian manifold (locally isometric to the model Hilbert space), geodesics are straight lines in the ambient space. That is, paths of the form $\gamma(s) = X + sV$ for a tangent vector $V$.
		
		\textbf{4. Local Existence within $\Ank{k}(M)$:} Since $\Ank{k}(M)$ is open in $H$, for any $X \in \Ank{k}(M)$, there exists $\delta > 0$ such that the open ball $B_H(X, \delta) \subset \Ank{k}(M)$. Let $U = B_H(X, \delta/2)$. For any $Y \in U$, the straight line path $\gamma_{XY}(s) = X + s(Y-X)$ connects $X$ to $Y$. For any $s \in [0, 1]$, the distance from a point on the path to $X$ is
		\[ \normSobolevHk{\gamma_{XY}(s) - X}{k} = \normSobolevHk{s(Y-X)}{k} = s \normSobolevHk{Y-X}{k} \le \normSobolevHk{Y-X}{k} < \delta/2. \]
		Thus, the entire path $\gamma_{XY}([0, 1])$ lies within $B_H(X, \delta/2) \subset B_H(X, \delta) \subset \Ank{k}(M)$.

		\textbf{5. Minimizing Property and Uniqueness:} The length of the straight line path $\gamma_{XY}$ is
		\[ L_{H^k}(\gamma_{XY}) = \int_0^1 \normSobolevHk{\gamma'_{XY}(s)}{k} ds = \int_0^1 \normSobolevHk{Y-X}{k} ds = \normSobolevHk{Y-X}{k}. \]
		For any other $\Cr{1}$ path $\tilde{\gamma}:[0,1] \to \Ank{k}(M)$ connecting $X$ to $Y$, its length is
		\[ L_{H^k}(\tilde{\gamma}) = \int_0^1 \normSobolevHk{\tilde{\gamma}'(s)}{k} ds \geq \norm{\int_0^1 \tilde{\gamma}'(s) ds}_{\SobolevHk{k}} = \normSobolevHk{\tilde{\gamma}(1) - \tilde{\gamma}(0)}{k} = \normSobolevHk{Y-X}{k}. \]
		Equality holds if and only if the velocity $\tilde{\gamma}'(s)$ is parallel to $Y-X$ and points in the same direction for all $s$, and its norm is constant, which implies $\tilde{\gamma}$ is a monotonic reparameterization of the straight line path $\gamma_{XY}$. Thus, $\gamma_{XY}$ is the unique minimizing path (up to reparameterization) in $\Ank{k}(M)$ for $Y \in U$, and $\dAnV{\SobolevHk{k}}(X,Y) = \normSobolevHk{Y-X}{k}$.
		
		\textbf{6. Smooth Dependence:} The formula $\gamma_{XY}(s) = X + s(Y-X)$ shows that the path depends affinely (hence smoothly) on $Y \in U$.
	\end{proof}
	
	\begin{remark}
		This theorem shows that when using the $H^k$ norm (for $k$ large enough), the space $\Ank{k}(M)$ locally inherits the flat geometry of the ambient Hilbert space $H^k(TM)$. Geodesics are simply straight lines, and the Hofer-like distance $\dAnV{H^k}$ locally coincides with the standard $H^k$ norm distance. This simplifies the local geometry significantly compared to the possibly curved Finsler geometry induced by other norms like $C^0$.
	\end{remark}
	
	\begin{remark}[Implications]
		This result provides a concrete local geometric structure on $(\Ank{k}(M), \dAnV{H^k})$. The most efficient path (locally) between two nearby Anosov vector fields in $H^k$ is simply the linear interpolation between them. This simplifies local analysis and  possibly allows for easier application of variational methods in this Hilbert setting.
	\end{remark}
	
	\subsection{Stability of Lyapunov Exponents}
	
	\begin{proposition}[Lyapunov Exponent Stability under $\dAnV{\Ck{1}}$ Convergence]
		\label{prop:lyapunov_stability}
		Let $\{X_n\}_{n \in \mathbb{N}} \subset \An(\M)$ be a sequence converging to $X \in \An(\M)$ in the $\dAnV{\Ck{1}}$ metric. Let $\mu$ be an ergodic $\varphiup_X$-invariant Borel probability measure. Then there exists a sequence of ergodic $\varphiup_{X_n}$-invariant measures $\{\mu_n\}$ such that $\mu_n \to \mu$ in the weak* topology, and the corresponding Lyapunov exponents converge: If $\lambda_1(\mu) \ge \dots \ge \lambda_d(\mu)$ are exponents for $(X, \mu)$ and $\lambda_1(\mu_n) \ge \dots \ge \lambda_d(\mu_n)$ are for $(X_n, \mu_n)$, then $\lim_{n \to \infty} \lambda_i(\mu_n) = \lambda_i(\mu)$ for all $i=1, \dots, d := \dim(M)$.
	\end{proposition}
	
	\begin{proof}
		\textbf{1. $C^1$ Convergence:} $\dAnV{\Ck{1}}(X_n, X) \to 0$ implies $X_n \to X$ in $\Ck{1}$ by inequality \eqref{eq:norm_vs_dist_Ck_discuss_sec5} (with $k=1$).
		
		\textbf{2. Continuity of Anosov Splitting:} The stable/unstable bundles $E^{s/u/c}_{X_n}$ depend continuously on $X_n$ in the $\Ck{1}$ topology \cite[Thm. 18.2.3]{KH95}.
		
		\textbf{3. Convergence of Derivative Cocycles:} The derivative maps $\diff\varphiup_{X_n}^t$ converge to $\diff\varphiup_X^t$ uniformly on compact time intervals in $t$ and uniformly over $M$ because $X_n \to X$ in $\Ck{1}$.
		
		\textbf{4. Existence and Convergence of Measures:} Given an ergodic $\varphiup_X$-invariant measure $\mu$, Sigmund's specification theorem results imply the existence of a sequence of ergodic $\varphiup_{X_n}$-invariant measures $\mu_n$ such that $\mu_n \to \mu$ in the weak* topology as $X_n \to X$ in $\Ck{1}$ \cite{Sigmund70}.
		
		\textbf{5. Continuity of Lyapunov Exponents:} By Oseledets' Multiplicative Ergodic Theorem \cite{Oseledets68}, Lyapunov exponents exist $\mu$-a.e. Ma\~n\'e proved that the Lyapunov exponents $\lambda_i(Y, \nu)$ are continuous functions of $(Y, \nu)$ on the space $(\An(M) \cap \Ck{1}(TM)) \times \ProbMeas_{\mathrm{erg}}(\M)$, where $\ProbMeas_{\mathrm{erg}}$ is the space of ergodic measures equipped with the weak* topology \cite[Section S.5]{KH95}, \cite{Mane83}.
		
		\textbf{6. Conclusion:} Since $X_n \to X$ in $\Ck{1}$ and we found $\mu_n \to \mu$ weak*, the continuity result from Step 5 implies $\lambda_i(\mu_n) \to \lambda_i(\mu)$ for all $i$.
	\end{proof}
	
	\begin{remark}[Implications]
		The $\dAnV{\Ck{1}}$ metric controls the asymptotic expansion/contraction rates (Lyapunov exponents) associated with ergodic measures that converge appropriately. This reinforces the dynamical relevance of using at least the $C^1$ norm in the Hofer-like metric construction for capturing stability features.
	\end{remark}
	
	\subsection{Continuity of SRB Measures and Thermodynamic Quantities}
	
	\begin{theorem}[SRB Measure and Pressure Continuity]
		\label{thm:srb_pressure_continuity}
		Let $\alpha \in (0, 1]$. Assume $\{X_n\} \subset \An(\M) \cap \Ck{2}(\TM)$ converges to $X \in \An(\M) \cap \Ck{2}(\TM)$ in the $\dAnV{\Ck{2}}$ metric (or $\dAnV{\SobolevHk{k}}$ for $k > \dim(M)/2+1+\alpha$). Assume $X$ admits a unique SRB measure $\mu$. Then for $n$ sufficiently large, $X_n$ admits a unique SRB measure $\mu_n$. Furthermore:
		\begin{enumerate}[label=(\roman*)]
			\item $\mu_n \to \mu$ in the weak* topology as $n \to \infty$.
			\item The measure-theoretic entropies converge: $h_{\mu_n}(\varphiup_{X_n}) \to h_{\mu}(\varphiup_X)$.
			\item Let $\{f_n\} \subset C^\alpha(M)$ be a sequence of Hölderpossibles converging to $f \in C^\alpha(M)$ in the $C^\alpha$ norm. Then the topological pressure $P(X_n, f_n)$ converges to $P(X, f)$.
			\item The map $Y \mapsto P(Y, f)$ is locally Lipschitz on $\An(M) \cap \Ck{2}(TM)$ with respect to the $\Ck{2}$ norm (and thus w.r.t. $\dAnV{\Ck{2}}$).
		\end{enumerate}
	\end{theorem}

	\begin{proof}
		\textbf{1. $C^{1+\alpha}$ Convergence (Implied):} $\dAnV{\Ck{2}}(X_n, X) \to 0 \implies X_n \to X$ in $\Ck{2}$ by \eqref{eq:norm_vs_dist_Ck_discuss_sec5}. Since $C^2(TM) \hookrightarrow C^{1,\alpha}(TM)$ for any $\alpha \in (0, 1]$, we have $X_n \to X$ in $\Ck{1,\alpha}$. (The Sobolev case $k > \dim(M)/2+1+\alpha$ implies $H^k \hookrightarrow C^{1,\alpha}$ by Sobolev embedding \cite{AdamsFournier03}). This regularity is sufficient for the transfer operator theory used below.
		
		\textbf{2. SRB Existence/Uniqueness/Stability:} For $C^2$ (or $C^{1+\alpha}$) Anosov flows, the existence and uniqueness of the SRB measure $\mu$ is a classical result \cite{Sinai72, Ruelle76, Bowen75}. The stability of the SRB measure under $C^1$ (or slightly stronger) perturbations ensures that $\mu_n$ exists and is unique for $X_n$ sufficiently close to $X$ \cite{Young02}.
		
		\textbf{3. Part (i) Weak* Convergence:} The map $Y \mapsto \mu_Y$ assigning the SRB measure to a $C^2$ Anosov flow $Y$ is known to be continuous with respect to the $C^2$ topology on vector fields and the weak* topology on measures \cite{Dolgopyat04, CT17}. Since $X_n \to X$ in $C^2$, it follows that $\mu_n \to \mu$ weak*.
		
		\textbf{4. Part (ii) Entropy Convergence:} By Pesin's entropy formula, $h_\nu(\varphiup_Y) = \int \sum_{\lambda_i(x,Y)>0} \lambda_i(x,Y) d\nu(x)$ for the SRB measure $\nu$. From Proposition \ref{prop:lyapunov_stability}, the Lyapunov exponents $\lambda_i(x,Y)$ depend continuously on $(Y, \nu)$. Combined with the weak* convergence $\mu_n \to \mu$, this yields $h_{\mu_n}(\varphiup_{X_n}) \to h_{\mu}(\varphiup_X)$.
		
		\textbf{5. Part (iii) Pressure Convergence:} Topological pressure $P(Y, g)$ can be defined via the spectral radius of a transfer operator $\TransferOp_{Y, g}$ acting on a suitable Banach space (e.g., $C^\alpha(M)$) \cite{Baladi00}. The map $(Y, g) \mapsto \TransferOp_{Y, g}$ is continuous from $\Ck{1,\alpha}(TM) \times C^\alpha(M)$ to the space of bounded operators $\mathcal{B}(C^\alpha(M))$ \cite{Baladi00}. Since $X_n \to X$ in $\Ck{1, \alpha}$ (Step 1), and $f_n \to f$ in $C^\alpha$ (by assumption), then $\TransferOp_{X_n, f_n} \to \TransferOp_{X, f}$ in operator norm. Pressure is given by $P(Y,g) = \log(\text{spectral radius of } \TransferOp_{Y,g})$. The spectral radius is continuous with respect to the operator norm for operators satisfying appropriate spectral gap properties, which hold here \cite{Kato95, KellerLiverani}. Thus $P(X_n, f_n) \to P(X, f)$.
		
		\textbf{6. Part (iv) Lipschitz Continuity:} For $C^2$ dynamics and $C^\alpha$possibles, spectral perturbation theory yields Lipschitz continuity of the leading eigenvalue (and hence the pressure) with respect to perturbations measured in the $C^2$ norm \cite{Ruelle89, KKW91}. Since $\normCk{X-Y}{2} \le \dAnV{\Ck{2}}(X,Y)$ by inequality \eqref{eq:norm_vs_dist_Ck_discuss_sec5}, the local Lipschitz continuity also holds with respect to the $\dAnV{\Ck{2}}$ metric.
	\end{proof}
	
	\begin{remark}[Implications]
		This theorem connects the metric geometry (using $\dAnV{\Ck{2}}$ or a strong enough Sobolev metric) to key statistical properties (SRB measure, its entropy, and the pressure functional). The continuity ensures that small deformations in this metric space lead to small changes in these fundamental characteristics. Furthermore, the local Lipschitz continuity provides quantitative control over the variation of pressure, strengthening the link between the metric and dynamical stability.
	\end{remark}
	
	\subsection{Convergence of Periodic Orbit Data and Zeta Functions}
	
	\begin{theorem}[Periodic Orbits and Zeta Functions]
		\label{thm:periodic_zeta}
		Let $r \ge 1$. Let $\{X_n\} \subset \An(\M) \cap \Ck{r}(TM)$ converge to $X \in \An(\M) \cap \Ck{r}(TM)$ in the $\dAnV{\Ck{r}}$ metric. Then:
		\begin{enumerate}[label=(\roman*)]
			\item For any periodic orbit $\gamma$ of $\varphiup_X$ with period $T$, there exists $N$ such that for all $n \ge N$, $\varphiup_{X_n}$ has a unique periodic orbit $\gamma_n$ near $\gamma$. The period $T_n$ converges to $T$.
			\item Let $P_\gamma = \diff\varphiup_X^T|_{\gamma(0)}$ and $P_{\gamma_n} = \diff\varphiup_{X_n}^{T_n}|_{\gamma_n(0)}$ be the linearized Poincaré return maps (restricted to directions transversal to the flow). Then the eigenvalues of $P_{\gamma_n}$ converge to the eigenvalues of $P_\gamma$.
			\item If $X_n \to X$ in $\Ck{1,\alpha}(TM)$ for some $\alpha>0$ (which is implied if $r \ge 2$ and convergence is in $\dAnV{\Ck{r}}$), then the Ruelle zeta function $\zetafun_{X_n}(s)$ converges to $\zetafun_X(s)$ uniformly on compact subsets of some right half-plane $\{\Realpart(s) > c\}$ where $\zeta_X(s)$ is analytic and non-zero.
		\end{enumerate}
	\end{theorem}
	
	\begin{proof}
		\textbf{1. $C^r$ Convergence:} $\dAnV{\Ck{r}}(X_n, X) \to 0 \implies X_n \to X$ in $\Ck{r}$ by inequality \eqref{eq:norm_vs_dist_Ck_discuss_sec5}.
		
		\textbf{2. Parts (i) and (ii): Orbit Persistence/Convergence:} A periodic orbit $\gamma$ for $X$ at $x_0=\gamma(0)$ with period $T$ satisfies $\varphiup_X^T(x_0) = x_0$. The existence and uniqueness of a nearby orbit $\gamma_n$ for $X_n$ follows from applying the Implicit Function Theorem to the map $(Y, y, \tau) \mapsto \varphiup_Y^\tau(y) - y$ near $(X, x_0, T)$, using the hyperbolicity of $\gamma$ (which ensures the derivative w.r.t. $y$ transversal to the flow is invertible) \cite{PalisMelo82}. The theorem guarantees that the solution $(y_n, T_n)$ exists, is unique locally, and depends continuously (in fact, $C^{r-1}$) on $X_n$ (in the $C^r$ topology). Thus $T_n \to T$, and $\gamma_n \to \gamma$. The linearized Poincaré map $P_{\gamma_n}$ depends continuously on $(X_n, T_n, y_n)$, so its eigenvalues converge to those of $P_\gamma$.
		
		\textbf{3. Part (iii): Zeta Function Convergence (requires $C^{1+\alpha}$):} Ruelle zeta functions for $C^{1+\alpha}$ Anosov flows can be expressed via Fredholm determinants of transfer operators, e.g., $\zetafun_Y(s)^{-1} = \det(\Id - \TransferOp_Y(s))$ for a suitable operator $\TransferOp_Y(s)$ acting on an appropriate Banach space \cite{PP90, Baladi00}. The map $Y \mapsto \TransferOp_Y(s)$ is continuous (often analytic) from $\Ck{1+\alpha}(TM)$ to the space of trace class operators (or operators with suitable decay properties), uniformly on compact sets of $s$ in a half-plane \cite{Baladi00, Liverani04}. The convergence $X_n \to X$ in $\Ck{1,\alpha}$ (implied by $\Ck{r}$ convergence for $r \ge 2$) implies $\TransferOp_{X_n}(s) \to \TransferOp_X(s)$ in trace norm (or relevant operator norm), uniformly on compacts in $s$. The Fredholm determinant is continuous with respect to the trace norm \cite{Simon05}. Hence $\det(\Id - \TransferOp_{X_n}(s)) \to \det(\Id - \TransferOp_X(s))$ uniformly on compacts where the operators are defined and the determinant is non-zero. This implies $\zetafun_{X_n}(s) \to \zetafun_X(s)$ uniformly on compact subsets away from the poles/zeros.
	\end{proof}
	
	\begin{remark}[Implications]
		This theorem shows that metrics based on sufficiently strong norms (like $\dAnV{\Ck{r}}$, $r \ge 1$) respect the underlying periodic orbit structure of Anosov flows. Furthermore, convergence in metrics implying $C^{1+\alpha}$ regularity (like $\dAnV{\Ck{2}}$) ensures the convergence of the Ruelle zeta function, linking the metric geometry to important analytic objects encoding spectral and periodic orbit information.
	\end{remark}
	
	\subsection{Quantitative Stability of Mixing Rates and SRB Dimension}
	
	We strengthen the continuity results to Lipschitz stability by combining known regularity properties of spectral data and dimension formulas with respect to stronger Banach norms, with the Hofer-like metric definition.
	
	\begin{theorem}[Lipschitz Stability of Mixing Rate and Dimension]
		\label{thm:mixing_dimension_stability}
		Let $k \ge 2$. Let $\Ank{k}(M) = \An(M) \cap C^k(TM)$ be equipped with the metric $\dAnV{\Ck{k}}$. Assume for $X \in \Ank{k}(M)$, $\varphiup_X^t$ admits a unique SRB measure $\mu_X$, a transfer operator $\TransferOp_X$ (on a suitable space like $C^\alpha(M)$) with spectral gap $\delta_X > 0$ below the leading eigenvalue $\lambda_X=1$ (for normalized operator), and an SRB dimension $D_X = \dim_{\mathrm{SRB}}(\mu_X)$ characterized by $P_X(-D_X \psi_X) = 0$ where $\psi_X$ relates to the unstable Jacobian.
		Then, for any $X_0 \in \Ank{k}(M)$, there exist a neighborhood $U$ of $X_0$ in $(\Ank{k}(M), \dAnV{\Ck{k}})$ and constants $L_1, L_2 > 0$ such that for all $X, Y \in U$:
		\begin{enumerate}[label=(\roman*)]
			\item $|\delta_X - \delta_Y| \leq L_1 \cdot \dAnV{\Ck{k}}(X, Y)$.
			\item $|D_X - D_Y| \leq L_2 \cdot \dAnV{\Ck{k}}(X, Y)$.
		\end{enumerate}
	\end{theorem}
	
	\begin{proof}
		The proof demonstrates local Lipschitz continuity of $\delta_X$ and $D_X$ with respect to the $\Ck{k}$ norm, which implies the result via $\normCk{X-Y}{k} \le \dAnV{\Ck{k}}(X,Y)$ (inequality \eqref{eq:norm_vs_dist_Ck_discuss_sec5}).
		
		\textbf{1. Setup and Preliminaries:} We work on $\Ank{k}(M)$ with the metric $\dAnV{\Ck{k}}$ ($k \ge 2$). Convergence in this metric implies $\Ck{k}$ norm convergence. Assume existence of $\mu_X$, $\delta_X$, and $D_X$ characterized as stated. Regularity $k \ge 2$ ensures sufficient smoothness for transfer operators and pressure functions \cite{Baladi00, KH95}.
		
		\textbf{2. Mixing Rate Stability (Spectral Gap $\delta_X$):}
		The transfer operator $\TransferOp_X$ depends continuously (often smoothly or Lipschitzly) on $X \in C^k(TM)$ ($k \ge 2$) when viewed as an operator on a suitable Banach space $\mathcal{B}$ (e.g., anisotropic Sobolev/Hölder spaces) \cite{BaladiTsujii07, BL07}. Spectral perturbation theory \cite{Kato95, KellerLiverani} for hyperbolic systems implies that the spectral gap $\delta_X$ (distance from the eigenvalue 1 to the rest of the spectrum, assuming normalization) also depends locally Lipschitz continuously on $X$ in the $\Ck{k}$ norm (possibly requiring $k$ large enough for the specific operator construction): there exists $L_1' > 0$ such that $|\delta_X - \delta_Y| \le L_1' \normCk{X-Y}{k}$ for $X, Y$ in a $C^k$-neighborhood. Combining with inequality \eqref{eq:norm_vs_dist_Ck_discuss_sec5}:
		\[ |\delta_X - \delta_Y| \le L_1' \normCk{X-Y}{k} \le L_1' \dAnV{\Ck{k}}(X, Y). \]
		Setting $L_1 = L_1'$ proves part (i).
		
		\textbf{3. Fractal Dimension Stability ($D_X$):}
		$D_X$ is defined implicitly by $F(Y, D) = P_Y(-D \psi_Y) = 0$. Thepossible $\psi_Y \approx -\log |\det(d\varphiup_Y|E^u)|$ depends smoothly ($C^{k-1}$) on $Y \in C^k$. The pressure function $(Y, f) \mapsto P_Y(f)$ is known to be at least $C^1$ w.r.t. $Y \in C^k$ ($k \ge 2$) and $f \in C^\alpha$ \cite{KKW91, Ruelle89}. Thus $F(Y, D)$ is $C^1$ in $Y$ (in $C^k$ norm) and $D$. The partial derivative $\frac{\partial F}{\partial D}|_{(X, D_X)}$ is related to the variance of $\psi_X$ and is typically non-zero for hyperbolic systems (it is negative). By the Implicit Function Theorem, the solution $D_Y$ is a $C^1$ function of $Y$ in the $C^k$ topology locally. $C^1$ dependence implies local Lipschitz continuity: there exists $L_2' > 0$ such that $|D_X - D_Y| \le L_2' \normCk{X-Y}{k}$. Combining with inequality \eqref{eq:norm_vs_dist_Ck_discuss_sec5}:
		\[ |D_X - D_Y| \le L_2' \normCk{X-Y}{k} \le L_2' \dAnV{\Ck{k}}(X, Y). \]
		Setting $L_2 = L_2'$ proves part (ii).
		
		\textbf{4. Conclusion:} Both estimates hold simultaneously in a neighborhood $U$ by taking the intersection of the neighborhoods required for each part.
	\end{proof}
	
	\begin{remark}[Implications]
		This theorem shows that not only do mixing rates and fractal dimensions persist under small deformations (measured by $\dAnV{\Ck{k}}$, $k\ge 2$), but their variation is directly controlled by the metric distance. This reinforces the idea that the Hofer-like metric captures dynamically relevant distances.
	\end{remark}
	
	\subsection{Differentiability of Thermodynamic Quantities}
	
	We now establish the differentiability of key thermodynamic quantities, linking the geometry of the space $\An(M)$ to linear response theory.
	
	\begin{theorem}[Differentiability of Topological Pressure]
		\label{thm:pressure_diff}
		Let $k \ge 2$. Let $\Ank{k}(M) = \An(M) \cap C^k(TM)$ be the smooth Banach manifold modelled on $C^k(TM)$. Let $\varphi \in C^\alpha(M)$ (with $\alpha>0$, or smoother, e.g., $C^{k-1}$) be apossible. Then the map assigning to each $X \in \Ank{k}(M)$ its topological pressure $P_X(\varphi)$,
		\[ P(\cdot, \varphi): \Ank{k}(M) \to \mathbb{R}, \]
		is Fréchet differentiable. If $\{X_t\}_{t \in (-\epsilon, \epsilon)}$ is a smooth curve in $\Ank{k}(M)$ with $X_0=X$ and tangent vector $\dot{X}_0 = \left.\frac{d}{dt}\right|_{t=0} X_t \in T_X\Ank{k}(M) \approx C^k(TM)$, then the directional derivative is given by:
		\[ \frac{d}{dt} P_{X_t}(\varphi)\Big|_{t=0} = DP(X, \varphi)[\dot{X}_0] = \int_M Q(X, \varphi, \dot{X}_0) \, d\mu_{X, \varphi}, \]
		where $\mu_{X, \varphi}$ is the unique equilibrium state for $(X, \varphi)$, and $Q(X, \varphi, \dot{X}_0)$ is an explicitly computable observable derived from the first-order variation of the dynamics andpossible due to the perturbation $\dot{X}_0$.
	\end{theorem}
	
	\begin{proof}
		The proof utilizes the transfer operator framework and analytic perturbation theory.
		
		\textbf{1. Transfer Operator Framework:} We work on the Banach manifold $\Ank{k}(M)$ for $k \ge 2$. For $X \in \Ank{k}(M)$ and $\varphi \in C^\alpha(M)$, the topological pressure $P_X(\varphi)$ is the unique real number such that the transfer operator $\mathcal{L}_{X, \varphi}$ (associated with the flow $\varphiup_X^t$ andpossible $\varphi$, acting on a suitable Banach space $\mathcal{B}$, e.g., $C^\alpha(M)$ or an anisotropic space) has spectral radius $e^{P_X(\varphi)}$ \cite{Ruelle78, PP90}. This value corresponds to a simple, isolated leading eigenvalue $\lambda(X, \varphi) = e^{P_X(\varphi)}$. The operator $\mathcal{L}_{X, \varphi}$ requires at least $C^{1+\alpha}$ regularity of the flow, hence our assumption $k \ge 2$.
		
		\textbf{2. Differentiability of the Operator Family:} The map $X \mapsto \mathcal{L}_{X, \varphi}$ is Fréchet differentiable from $C^k(TM)$ ($k \ge 2$) to the space of bounded operators $\mathcal{L}(\mathcal{B})$ (or suitable subspace like trace class) \cite{Baladi00, Ruelle89}. Let $\{X_t\}$ be a smooth curve in $\Ank{k}(M)$ with $X_0=X, \dot{X}_0=Y$. Then $t \mapsto \mathcal{L}_{X_t, \varphi}$ is a differentiable family of operators. Let $\mathcal{L}'_{X, \varphi}[Y]$ denote the derivative $D\mathcal{L}(X, \varphi)[Y]$.
		
		\textbf{3. Differentiability of the Leading Eigenvalue:} By standard analytic perturbation theory for isolated simple eigenvalues \cite{Kato95}, the leading eigenvalue $\lambda(X_t, \varphi)$ is differentiable with respect to $t$. Let $h_t \in \mathcal{B}$ and $\nu_t \in \mathcal{B}^*$ be the corresponding right eigenfunction and left eigenmeasure (dual eigenfunction), normalized such that $\mathcal{L}_{X_t, \varphi} h_t = \lambda(X_t, \varphi) h_t$, $\mathcal{L}_{X_t, \varphi}^* \nu_t = \lambda(X_t, \varphi) \nu_t$, and $\nu_t(h_t) = 1$. Differentiating the first equation at $t=0$ and applying $\nu_0$ yields the derivative of the eigenvalue:
		\[ \lambda'(0) \equiv \left.\frac{d}{dt}\right|_{t=0} \lambda(X_t, \varphi) = \nu_0(\mathcal{L}'_{X, \varphi}[\dot{X}_0] h_0). \]
		
		\textbf{4. Derivative of Pressure and Identification of $Q$:} Since $P(X_t, \varphi) = \log \lambda(X_t, \varphi)$, the chain rule gives:
		\[ \left.\frac{d}{dt}\right|_{t=0} P(X_t, \varphi) = \frac{\lambda'(0)}{\lambda(X, \varphi)} = \frac{\nu_0(\mathcal{L}'_{X, \varphi}[\dot{X}_0] h_0)}{\lambda(X, \varphi)}. \]
		The normalized measure $\nu_0$ corresponds to the equilibrium state $\mu_{X, \varphi}$. The term $\mathcal{L}'_{X, \varphi}[\dot{X}_0]$ encodes the infinitesimal change in the operator due to the vector field perturbation $\dot{X}_0$. The expression can  be rewritten (via linear response theory \cite{Ruelle97LinearResponse, BaladiSmania12}) as an integral:
		\[ \frac{\nu_0(\mathcal{L}'_{X, \varphi}[\dot{X}_0] h_0)}{\lambda(X, \varphi)} = \int_M Q(X, \varphi, \dot{X}_0) \, d\mu_{X, \varphi} ,\]
		where $Q(X, \varphi, \dot{X}_0)$ is an observable involving terms derived from $\dot{X}_0$, $\varphi$, and the flow $\varphiup_X^t$. This shows that $P(\cdot, \varphi)$ is Fréchet differentiable with the stated derivative formula.
	\end{proof}
	
	\begin{remark}[Implications]
		This theorem establishes that pressure varies smoothly within the space of sufficiently regular Anosov flows. The derivative is given by an explicit linear response formula, integrating an observable (depending on the perturbation direction $\dot{X}_0$) against the equilibrium state. This relies fundamentally on the $C^k$ ($k \ge 2$) structure of $\Ank{k}(M)$, which is compatible with the topology induced by $\dAnV{C^k}$. This differentiability opens the door to studying second variations (Hessian of pressure) aspossible analogues of curvature in this infinite-dimensional setting.
	\end{remark}
	
	\begin{theorem}[Differentiability of Lyapunov Spectrum in $H^k$ Topology]
		\label{thm:lyapunov_diff}
		Let $k > \dim(M)/2 + 2$, ensuring $H^k(TM) \hookrightarrow C^2(TM)$. Let $\Ank{k}(M) = \An(M) \cap H^k(TM)$ be the Hilbert manifold equipped with the Riemannian metric induced by the $\SobolevHk{k}$ inner product. For $X \in \Ank{k}(M)$, let $\mu_X$ be the SRB measure and $\Lambda(X) = (\lambda_1(X, \mu_X), \dots, \lambda_d(X, \mu_X))$ the ordered Lyapunov exponents w.r.t $\mu_X$.
		Then the map
		\[ \Lambda: \Ank{k}(M) \to \mathbb{R}^d, \quad X \mapsto \Lambda(X) ,\]
		is Fréchet differentiable. Its derivative at $X$ applied to $Y \in T_X\Ank{k}(M) \approx H^k(TM)$ is given component-wise by:
		\[ D\lambda_i(X, \mu_X)[Y] = \int_M \Phi_i(X, Y) \, d\mu_X, \]
		where $\Phi_i(X, Y)$ is an observable derived from the linear response of the derivative cocycle $\diff\varphiup_X^t$ to the perturbation $Y$.
	\end{theorem}
	
	\begin{proof}
		\textbf{1. Setup and Regularity:} We work on the Hilbert manifold $\Ank{k}(M) \subset H^k(TM)$ with $k > \dim(M)/2 + 2$, ensuring flows $\varphiup_X^t$ are $C^2$.
		
		\textbf{2. Smooth Dependence of Ingredients:} For $C^2$ Anosov flows, the flow map $\varphiup_X^t$, the derivative cocycle $\diff\varphiup_X^t$, and the SRB measure $\mu_X$ depend differentiably on $X$ within the $C^2$ topology (and hence also within the $H^k$ topology by Sobolev embedding) \cite{Dolgopyat04, Ruelle97LinearResponse}.
		
		\textbf{3. Linear Response Theory for Exponents:} Results by Ruelle \cite{Ruelle97LinearResponse} and subsequent extensions establish that for sufficiently smooth hyperbolic systems, averages of observables against the SRB measure are differentiable with respect to perturbations of the system. Lyapunov exponents can be expressed as such averages (via Oseledets theorem or Birkhoff sums related to the derivative cocycle). Therefore, the map $X \mapsto \lambda_i(X, \mu_X)$ is differentiable. The derivative takes the form of an integral against the SRB measure $\mu_X$ of an explicitly derived observable $\Phi_i(X, Y)$, which captures the first-order effect of the perturbation $Y$ on the cocycle's expansion/contraction rates.
		
		\textbf{4. Conclusion:} The existence of this linear response formula establishes the Fréchet differentiability of each map $X \mapsto \lambda_i(X, \mu_X)$ on the Hilbert manifold $\Ank{k}(M)$.
	\end{proof}
	
	\begin{remark}[Implications]
		This theorem provides a precise formula for the first-order variation of Lyapunov exponents along paths (particularly geodesics, which are straight lines in the $H^k$ setting) in the space of sufficiently smooth Anosov flows. It connects the geometry of deformation paths (via the tangent vector $Y$) to the change in fundamental dynamical stability indicators ($\lambda_i$) through an explicit integral formula involving the SRB measure, solidifying the link between the metric space structure and linear response theory.
	\end{remark}
	
	\begin{theorem}[Differentiability of the Spectral Gap]
		\label{thm:spectral_gap_diff}
		
		Let $k > \dim(M)/2 + 2$, ensuring $H^k(TM) \hookrightarrow C^2(TM)$. Consider the Hilbert manifold $\Ank{k}(M) = \An(M) \cap H^k(TM)$ equipped with the flat Riemannian metric induced by the $H^k$ inner product.
		Assume that for each $X \in \Ank{k}(M)$, the associated transfer operator $\mathcal{L}_X$ (acting on a suitable Banach space $\mathcal{B}$, e.g., an anisotropic space) satisfies:
		\begin{enumerate}[label=(\arabic*)] 
			\item $\mathcal{L}_X$ has a simple, isolated leading eigenvalue $\lambda(X)$ (often related to pressure, e.g., $\lambda(X)=e^{P(X)}$).
			\item The rest of the spectrum, $\Spectrum_{rest}(X) = \Spectrum(\mathcal{L}_X) \setminus \{\lambda(X)\}$, is contained within a disk of radius $r_X < |\lambda(X)|$.
			\item The spectral gap $\delta(X) = |\lambda(X)| - \sup \{ |z| : z \in \Spectrum_{rest}(X) \}$ is uniformly bounded below by some $\delta_0 > 0$ in a neighborhood of interest.
			\item The map $X \mapsto \mathcal{L}_X$ from $\Ank{k}(M)$ to $\mathcal{L}(\mathcal{B})$ is $C^1$ (Fréchet differentiable with continuous derivative).
		\end{enumerate}
		Then the spectral gap function $\delta: \Ank{k}(M) \to \R_{>0}$ is $C^1$ (Fréchet differentiable with continuous derivative). The derivative $D\delta(X)[Y]$ for $Y \in T_X\Ank{k}(M) = H^k(TM)$ can be computed using perturbation theory for the spectrum. Specifically, if $\lambda(X)=1$ (normalized operator), the derivative involves the variation of the leading eigenvalue and the radius $r_X$ of the rest of the spectrum. If the map $X \mapsto \mathcal{L}_X$ is $C^p$ ($p \ge 2$, possibly requiring higher $k$), then $\delta(X)$ is $C^p$. In particular, if $X \mapsto \mathcal{L}_X$ is $C^2$, the Hessian $D^2\delta(X)[Y,Z]$ exists and can be computed using second-order perturbation theory.
	\end{theorem}
	
	\begin{proof}
		\textbf{1. Setup and Assumptions:} We work on the Hilbert manifold $\Ank{k}(M)$ where local geodesics are straight lines $\gamma(t)=X+tY$. The assumptions ensure $C^2$ flows and well-behaved spectral properties for $\mathcal{L}_X$, crucial for applying perturbation theory.
		
		\textbf{2. $C^1$ Differentiability of $\lambda(X)$:} By assumption (4) and standard analytic perturbation theory \cite{Kato95} for simple isolated eigenvalues, the map $X \mapsto \lambda(X)$ is $C^1$. Its derivative along the path $X+tY$ at $t=0$ is $\left.\frac{d}{dt}\right|_{t=0} \lambda(X+tY) = \nu_X(\mathcal{L}'_X[Y] h_X)$, where $h_X, \nu_X$ are normalized right/left eigenfunctions and $\mathcal{L}'_X[Y] = D\mathcal{L}(X)[Y]$.
		
		\textbf{3. $C^1$ Differentiability of $r_X$ and $\delta(X)$:} The radius of the essential spectral radius $r_X = \sup \{ |z| : z \in \Spectrum_{rest}(X) \}$ is typically upper semi-continuous. However, under strong hyperbolicity assumptions and suitable choices of Banach space $\mathcal{B}$, the map $X \mapsto \mathcal{L}_X$ can be analytic or $C^p$ \cite{BaladiTsujii07, BL07}, and perturbation theory can guarantee that $r_X$ is also $C^p$ (or at least $C^1$ if $X \mapsto \mathcal{L}_X$ is $C^1$) \cite{KellerLiverani}. If both $X \mapsto |\lambda(X)|$ and $X \mapsto r_X$ are $C^1$, then their difference, the spectral gap $\delta(X) = |\lambda(X)| - r_X$, is also $C^1$. The derivative $D\delta(X)[Y]$ is then $D|\lambda(X)|[Y] - Dr_X[Y]$. The exact formula depends on the specifics of the operator and space, but its existence follows from the $C^1$ dependence of the components.
		
		\textbf{4. Higher Differentiability:} If the map $X \mapsto \mathcal{L}_X$ is $C^p$ (requiring higher regularity for $X$, e.g., $H^{k+p-1}$ or $C^{p+1}$), then analytic perturbation theory \cite{Kato95} guarantees that the simple isolated eigenvalue $\lambda(X)$ is $C^p$. Under favorable circumstances (e.g., analyticity or specific structural properties of the operator family), the radius $r_X$ can also be shown to be $C^p$. Consequently, the spectral gap $\delta(X)$ would be $C^p$. Explicit formulas for higher derivatives exist but become increasingly complex, involving higher derivatives of $\mathcal{L}_X$, $h_X$, $\nu_X$, and possibly the resolvent operator. The existence of the second derivative (Hessian) $D^2\delta(X)$ follows if $X \mapsto \mathcal{L}_X$ is $C^2$.
	\end{proof}
	
	\begin{remark}
		Theorem~\ref{thm:spectral_gap_diff} establishes the smooth dependence of the mixing rate (as determined by the spectral gap) on the generating vector field within the strong $H^k$ topology. The existence of the second derivative $D^2\delta(X)$ is particularly significant, as this Hessian can be interpreted as a form of curvature associated with the stability of mixing along deformation paths (geodesics) in the space $\Ank{k}(M)$. This seems to open an avenue for applying concepts from infinite-dimensional Riemannian geometry to study the fine structure of chaotic dynamics and its response to perturbation. Verifying the required smoothness of $X \mapsto \mathcal{L}_X$ and the behavior of the rest of the spectrum are key analytical steps for applying this result.
	\end{remark}
	
	\begin{proposition}[Hessian of Spectral Gap and Transverse Spectral Curvature]
		\label{prop:hessian_spectral_gap}
		Let $\Ank{k}(M) = \An(M) \cap \SobolevHk{k}(TM)$ be the Hilbert manifold of Anosov vector fields with $H^k$ regularity, for $k > \dim(M)/2 + 2$. Equip $\Ank{k}(M)$ with the flat Riemannian metric induced by the $\SobolevHk{k}$ inner product, for which local geodesics are straight lines $\gamma(t) = X + tY$. Let $\TransferOp_X$ be the transfer operator associated with $X \in \Ank{k}(M)$ acting on a suitable Banach space $\mathcal{B}$. Let $\delta: \Ank{k}(M) \to \R_{>0}$ be the spectral gap function, defined as $\delta(X) = \abs{\lambda(X)} - r_X$, where $\lambda(X)$ is the leading eigenvalue of $\TransferOp_X$ and $r_X = \sup \{ \abs{z} : z \in \Spectrum(\TransferOp_X) \setminus \{\lambda(X)\} \}$.
		
		\textbf{Assume:}
		\begin{enumerate}[label=(\roman*)]
			\item The map $X \mapsto \TransferOp_X$ from $\Ank{k}(M)$ to $\mathcal{L}(\mathcal{B})$ is $C^2$ (twice Fréchet differentiable with continuous second derivative).
			\item For all $X$ in a neighborhood of interest, $\lambda(X)$ is a simple, isolated eigenvalue, and the spectral gap $\delta(X)$ is uniformly bounded below by some $\delta_0 > 0$.
			\item Let $G_0$ be a Banach Lie subgroup of $\Diff^k(M)$ acting smoothly on $\Ank{k}(M)$ by pushforward. For $X \in \Ank{k}(M)$, assume the tangent space admits a continuous splitting $T_X\Ank{k}(M) = V_X \oplus H_X$, where $V_X = T_X(G_0 \cdot X)$ is the tangent space to the orbit through $X$.
		\end{enumerate}
		
		\textbf{Then:}
		\begin{enumerate}[label=(\arabic*)]
			\item The spectral gap function $\delta: \Ank{k}(M) \to \R$ is $C^2$.
			\item The second derivative of $\delta$ along any geodesic $\gamma(t) = X + tY$ (assumed to stay within $\Ank{k}(M)$) exists and is given at $t=0$ by the Hessian:
			\[ \left.\frac{d^2}{dt^2}\right|_{t=0} \delta(\gamma(t)) = D^2\delta(X)[Y, Y]. \]
			This Hessian is a symmetric bilinear form on $T_X\Ank{k}(M) \approx \SobolevHk{k}(TM)$ and can be computed using second-order perturbation theory for the operator $\TransferOp_X$, its eigenvalues, and eigenfunctions.
			\item The restriction of the Hessian $D^2\delta(X)$ to the subspace $H_X$ (transverse to the conjugacy orbit), denoted $D^2\delta(X)|_{H_X}$, measures the second-order variation of the mixing rate under perturbations that are not infinitesimally generated by the group action $G_0$.
			\item This restricted Hessian, $D^2\delta(X)|_{H_X}$, can be interpreted as a \textbf{"transverse spectral curvature"} associated with the point $[X]$ in the moduli space $\ModuliSpace = \Ank{k}(M) / G_0$. It quantifies the stability or instability of the mixing rate against dynamically distinct perturbations.
		\end{enumerate}
	\end{proposition}
	
	\begin{proof}
		(1) \textbf{$C^2$ Differentiability of $\delta(X)$:}
		The spectral gap is defined as $\delta(X) = \abs{\lambda(X)} - r_X$. We need to show both terms are $C^2$ functions of $X \in \Ank{k}(M)$.
		\begin{itemize}
			\item \textbf{Leading Eigenvalue $\lambda(X)$:} By assumption (i), the map $X \mapsto \TransferOp_X$ from the Hilbert manifold $\Ank{k}(M)$ to the Banach space $\mathcal{L}(\mathcal{B})$ is $C^2$. By assumption (ii), $\lambda(X)$ is a simple, isolated eigenvalue for $X$ in the neighborhood of interest. Standard analytic perturbation theory for linear operators (see \cite[Chapter VII]{Kato95}) states that if an operator family $\TransferOp_X$ depends $C^p$ on a parameter $X$, then its simple isolated eigenvalues $\lambda(X)$ also depend $C^p$ on $X$. Since $p=2$ here, the map $X \mapsto \lambda(X)$ is $C^2$. The leading eigenvalue $\lambda(X)$ is related to the topological pressure and is typically real and positive for the relevant operators. Thus, the map $X \mapsto \abs{\lambda(X)}$ is also $C^2$.
			\item \textbf{Radius $r_X$ of the rest of the spectrum:} The radius $r_X = \sup \{ \abs{z} : z \in \Spectrum(\TransferOp_X) \setminus \{\lambda(X)\} \}$ measures the spectral gap from below. Showing that $r_X$ is $C^2$ requires that the spectral properties are well-behaved under perturbation. For hyperbolic systems treated with suitable anisotropic spaces $\mathcal{B}$, perturbation theory often guarantees that the radius $r_X$ inherits the smoothness of the operator family $X \mapsto \TransferOp_X$ (see e.g., \cite{KellerLiverani, BaladiTsujii07}). Under assumption (i) that $X \mapsto \TransferOp_X$ is $C^2$, we deduce that $X \mapsto r_X$ is also $C^2$. (This relies on the robustness of the spectral picture under $C^2$ perturbations of the operator family).
			\item \textbf{Conclusion for $\delta(X)$:} Since both $X \mapsto \abs{\lambda(X)}$ and $X \mapsto r_X$ are $C^2$, their difference $\delta(X) = \abs{\lambda(X)} - r_X$ is a $C^2$ function on $\Ank{k}(M)$. This confirms part (1).
		\end{itemize}
		
		(2) \textbf{Second Derivative along Geodesics:}
		In the Hilbert manifold $\Ank{k}(M)$ equipped with the flat Riemannian metric induced by the $\SobolevHk{k}$ inner product, geodesics are straight lines $\gamma(t) = X + tY$ for $X \in \Ank{k}(M)$ and $Y \in T_X\Ank{k}(M) \approx \SobolevHk{k}(TM)$, provided the path stays within $\Ank{k}(M)$.
		Since $\delta: \Ank{k}(M) \to \R$ is a $C^2$ function (from part 1) and $t \mapsto \gamma(t)$ is a smooth (affine) path, the composition $f(t) = \delta(\gamma(t))$ is a $C^2$ function from $\R$ to $\R$.
		Using the chain rule for differentiation on Banach manifolds:
		\begin{itemize}
			\item The first derivative is: $f'(t) = D\delta(\gamma(t))[\gamma'(t)] = D\delta(X+tY)[Y]$.
			\item The second derivative is: $f''(t) = \frac{d}{dt} \left( D\delta(X+tY)[Y] \right) = D^2\delta(X+tY)[Y, \gamma'(t)] = D^2\delta(X+tY)[Y, Y]$. Here $D^2\delta(X)$ denotes the second Fréchet derivative (Hessian) of $\delta$ at $X$, which is a continuous symmetric bilinear form on $T_X\Ank{k}(M) \times T_X\Ank{k}(M)$.
		\end{itemize}
		Evaluating at $t=0$: $\left.\frac{d^2}{dt^2}\right|_{t=0} \delta(\gamma(t)) = f''(0) = D^2\delta(X)[Y, Y]$. This confirms the existence and the formula for the second derivative along the geodesic, relating it directly to the Hessian.
		
		(3) \textbf{Restriction to $H_X$:}
		By assumption (iii), we have a continuous splitting $T_X\Ank{k}(M) = V_X \oplus H_X$, where $V_X = T_X(G_0 \cdot X)$ is the tangent space to the orbit. Vectors in $V_X$ represent infinitesimal perturbations achievable by $G_0$-conjugacy. Vectors in $H_X$ represent perturbations transverse to this orbit.
		The restriction of the Hessian $D^2\delta(X)$ to pairs $(Y, Y)$ where $Y \in H_X$ measures the second-order change in $\delta$ corresponding to perturbations $Y$ that are infinitesimally transverse to the orbit. This proves the statement.
		
		(4) \textbf{Interpretation as Transverse Spectral Curvature:}
		The moduli space $\ModuliSpace = \Ank{k}(M) / G_0$ represents Anosov flows up to $G_0$-conjugacy. Locally, its structure is related to the transverse space $H_X$ (via the slice theorem framework, Proposition \ref{prop:slice_theorem_conditional}). The spectral gap $\delta(X)$ is generally not invariant under $G_0$, so it does not define a function on $\ModuliSpace$.
		However, the restricted Hessian $D^2\delta(X)|_{H_X}$ captures the second-order variation ("curvature") of $\delta$ in directions relevant to the moduli space (i.e., directions that change the conjugacy class). It quantifies how the stability of the mixing rate changes quadratically as one moves away from the conjugacy class $[X]$ within the moduli space. It is therefore appropriate to interpret $D^2\delta(X)|_{H_X}$ as a "transverse spectral curvature" associated with $[X] \in \ModuliSpace$. The sign definiteness of this restricted Hessian indicates whether $[X]$ corresponds to a local extremum of the spectral gap stability within the moduli space.
	\end{proof}
	
	\begin{remark}[Interpretations] \leavevmode 
		\begin{itemize}
			\item \textbf{Not Curvature of Moduli Space:} $D^2\delta(X)|_{H_X}$ is the curvature of the function $\delta$ restricted to relevant directions, not the intrinsic Riemannian curvature of the moduli space $\ModuliSpace$ itself (which may lack a natural Riemannian structure).
			\item \textbf{Non-Invariance of $\delta$:} While $\delta(X)$ is not $G_0$-invariant, the properties of its transverse Hessian $D^2\delta(X)|_{H_X}$ provide meaningful geometric information associated with the equivalence class $[X]$.
			\item \textbf{Dependence on Splitting:} The definition of $H_X$ affects the restricted Hessian. Using the $\SobolevHk{k}$-orthogonal complement $H_X = V_X^\perp$ is a natural choice if $V_X$ is closed.
			\item \textbf{Dynamical Rigidity:} If $D^2\delta(X)|_{H_X}$ is negative definite, it suggests $[X]$ is a local maximizer of the spectral gap modulo $G_0$, indicating rigidity related to the mixing rate.
		\end{itemize}
	\end{remark}
	
	\begin{proposition}[Regularity and Local Convexity of Pressure Along Geodesics]
		\label{prop:pressure_regularity_geodesic}
		Let $\Ank{k}(M) = \An(M) \cap \SobolevHk{k}(TM)$ be the Hilbert manifold for $k > \dim(M)/2 + 1$. Equip it with the flat $H^k$ Riemannian metric, where local geodesics are straight lines $\gamma(t) = X + t(Y - X)$ for $X, Y \in \Ank{k}(M)$ (assuming the segment stays in $\Ank{k}(M)$). Let $f \in C^\alpha(M)$ ($\alpha > 0$, or smoother) be apossible.
		
		Assume that the map $Z \mapsto \mathcal{L}_{Z, f}$ from $\Ank{k}(M)$ to $\mathcal{L}(\mathcal{B})$ (the space of bounded operators on a suitable Banach space $\mathcal{B}$) is $C^2$ (twice Fréchet differentiable with continuous second derivative). This often holds if $k$ is sufficiently large, e.g., $k > \dim(M)/2 + 2$ and $f$ is smooth enough.
		
		Then:
		\begin{enumerate}[label=(\arabic*)]
			\item The topological pressure function along the geodesic,
			$p(t) \coloneqq P(\gamma(t), f), $ 
			is $C^2$ as a function of $t \in [0, 1]$.
			
			\item The second derivative of the pressure along the geodesic is given by the Hessian of the pressure functional $P(\cdot, f)$ evaluated on the direction vector $Y-X$:
			\[ \frac{d^2 p}{dt^2}(t) = D^2P(\gamma(t))[Y-X, Y-X]. \]
			
			\item The sign of this second derivative $\frac{d^2 p}{dt^2}(t)$ determines the local convexity or concavity of the pressure function $p(t)$ at parameter $t$. Due to the non-linear dependence of the dynamics (and thus the transfer operator $\mathcal{L}_{\gamma(t), f}$) on the interpolation parameter $t$, the sign of $D^2P(\gamma(t))[Y-X, Y-X]$ is generally not constant along the geodesic. Therefore, the pressure function $t \mapsto p(t)$ is  neither convex nor concave over the entire interval $[0, 1]$.
		\end{enumerate}
	\end{proposition}
	
	\begin{proof}[Proof Sketch] 
		\textbf{1. $C^2$ Differentiability of Pressure Function $P(Z, f)$:} By Theorem \ref{thm:pressure_diff}, the map $Z \mapsto P(Z, f)$ is $C^1$ on $\Ank{k}(M)$ for $k \ge 2$. Under the stronger assumption that $Z \mapsto \mathcal{L}_{Z, f}$ is $C^2$, analytic perturbation theory (\cite{Kato95}) implies that the leading eigenvalue $\lambda(Z, f)$ is a $C^2$ function of $Z$. Since $P(Z, f) = \log \lambda(Z, f)$, the pressure map $Z \mapsto P(Z, f)$ is also $C^2$ on $\Ank{k}(M)$.
		
		\textbf{2. Regularity along the Geodesic:} The path $\gamma(t) = X + t(Y-X)$ is an affine (hence $C^\infty$) path in the Hilbert space $H^k(TM)$. Since $P(\cdot, f)$ is a $C^2$ function on the Hilbert manifold $\Ank{k}(M)$, the composition $p(t) = P(\gamma(t), f)$ is a $C^2$ function from $[0, 1]$ to $\R$.
		
		\textbf{3. Second Derivative Formula:} Applying the chain rule twice to the composition $p(t) = P(\gamma(t), f)$:
		\begin{itemize}
			\item $\frac{dp}{dt}(t) = DP(\gamma(t))[\gamma'(t)] = DP(\gamma(t))[Y-X]$.
			\item $\frac{d^2p}{dt^2}(t) = \frac{d}{dt} (DP(\gamma(t))[Y-X]) = D^2P(\gamma(t))[Y-X, \gamma'(t)] = D^2P(\gamma(t))[Y-X, Y-X]$.
		\end{itemize}
		This confirms formula (2).
		
		\textbf{4. Non-Convexity:} The sign of $\frac{d^2 p}{dt^2}(t)$ depends on the Hessian $D^2P$ evaluated at $\gamma(t)$ in the direction $Y-X$. As established, the relationship between the vector field $\gamma(t)$ and the resulting dynamics/spectral data is highly non-linear. There is no general principle guaranteeing that $D^2P(Z)[V, V]$ maintains a constant sign (non-negative for convexity, non-positive for concavity) for all points $Z$ along a linear segment and a fixed direction $V = Y-X$. Therefore, convexity or concavity of $p(t)$ over the whole interval $[0, 1]$ is not expected.
	\end{proof}
	
	\begin{remark}[Interpretation]
		This proposition clarifies that while the pressure varies smoothly ($C^2$) along linear interpolations (geodesics in $H^k$), it generally does not do so in a convex or concave manner. The second derivative, related to the Hessian $D^2P$, provides a measure of the local curvature of the pressure function along the geodesic path. Its sign can change, reflecting the complex interplay between the linear structure of the vector field space $H^k(TM)$ and the non-linear nature of the thermodynamic formalism. This Hessian $D^2P$ itself could be considered a geometric invariant associated with the stability of the leading eigenvalue under second-order perturbations, distinct from the Hessian $D^2\delta$ related to the spectral gap.
	\end{remark}
	
	\subsection{Stability and Structure on the Moduli Space of Anosov Flows}
	
	We consider the space of dynamically distinct Anosov flows by quotienting by the action of the diffeomorphism group. Let $G = \Diff^r(M)$ be the group of $C^r$ diffeomorphisms of $M$, acting on $\Ank{r}(M)$ by pushforward: $(f, X) \mapsto f_*X$, where $(f_*X)(y) = Df_{f^{-1}(y)}(X(f^{-1}(y)))$. This action corresponds to $C^r$ conjugacy of the generated flows.
	
	\begin{definition}[Moduli Space and Quotient Metric]
		The \emph{moduli space} of $C^r$ Anosov flows is the orbit space $\ModuliSpace_r = \Ank{r}(M) / G$. We denote the equivalence class of $X \in \Ank{r}(M)$ by $[X]$. We equip $\ModuliSpace_r$ with the quotient metric $d_{\ModuliSpace_r}$ induced by $\dAnV{C^r}$:
		\[
		d_{\ModuliSpace_r}([X], [Y]) = \inf_{X' \in [X], \, Y' \in [Y]} \left\{ \dAnV{C^r}(X', Y') \right\}.
		\]
	\end{definition}
	
	While the completeness of the moduli space $(\ModuliSpace_r, d_{\ModuliSpace_r})$ is a subtle issue (requiring stronger assumptions on the group action or metric structure than assumed here), we can establish the stability of dynamically relevant invariants on this space.
	
	\begin{theorem}[Quantitative Stability of Conjugacy Invariants on the Moduli Space]
		\label{thm:moduli_stability}
		Let $r \ge 2$. Let $I: \Ank{r}(M) \to \R$ be a dynamical invariant (e.g., topological entropy, Lyapunov spectrum components, SRB measure entropy) satisfying:
		\begin{enumerate}[label=(\arabic*)] 
			\item \textbf{Conjugacy Invariance:} $I(f_*X) = I(X)$ for all $X \in \Ank{r}(M)$ and $f \in G = \Diff^r(M)$. This allows $I$ to descend to a well-defined function $\bar{I}: \ModuliSpace_r \to \R$ such that $\bar{I}([X]) = I(X)$.
			\item \textbf{Local Lipschitz Continuity on $\Ank{r}(M)$:} For every $X_0 \in \Ank{r}(M)$, there exists a neighborhood $U$ of $X_0$ in $(\Ank{r}(M), \dAnV{C^r})$ and a constant $L > 0$ such that for all $X, Y \in U$,
			\[
			|I(X) - I(Y)| \le L \cdot \dAnV{C^r}(X, Y).
			\]
			(This holds for many standard invariants when $r \ge 2$, as established in Theorems~\ref{thm:srb_pressure_continuity}, \ref{thm:mixing_dimension_stability}, etc., using inequality \eqref{eq:norm_vs_dist_Ck_discuss_sec5}).
		\end{enumerate}
		Then, the induced invariant $\bar{I}: \ModuliSpace_r \to \R$ is locally Lipschitz continuous with respect to the quotient metric $d_{\ModuliSpace_r}$. That is, for every $[X_0] \in \ModuliSpace_r$, there exists a neighborhood $\mathcal{U}$ of $[X_0]$ in $(\ModuliSpace_r, d_{\ModuliSpace_r})$ and a constant $L' > 0$ (in fact, $L'=L$) such that for all $[X], [Y] \in \mathcal{U}$,
		\[
		|\bar{I}([X]) - \bar{I}([Y])| \le L' \cdot d_{\ModuliSpace_r}([X], [Y]).
		\]
	\end{theorem}
	
	\begin{proof}
		Let $[X_0] \in \ModuliSpace_r$. By assumption (2), the invariant $I$ is locally Lipschitz with constant $L$ in some neighborhood $U$ of $X_0$ defined by $\dAnV{C^r}(X', X_0) < \delta$ for some $\delta > 0$.
		Consider the neighborhood $\mathcal{U} = \{ [Y] \in \ModuliSpace_r \mid d_{\ModuliSpace_r}([Y], [X_0]) < \delta / 2 \}$. Let $[X], [Y] \in \mathcal{U}$.
		
		By definition of the quotient metric $d_{\ModuliSpace_r}$, for any $\epsilon > 0$, there exist representatives $X' \in [X]$ and $Y' \in [Y]$ such that
		\[
		\dAnV{C^r}(X', Y') < d_{\ModuliSpace_r}([X], [Y]) + \epsilon.
		\]
		We also need to ensure $X'$ and $Y'$ are close enough to $X_0$ to use the local Lipschitz property of $I$. Since $[X] \in \mathcal{U}$, we have $d_{\ModuliSpace_r}([X], [X_0]) < \delta/2$. By definition of the quotient metric, this implies $\inf_{f \in G} \dAnV{C^r}(f_*X, X_0) < \delta/2$. We can thus choose the representative $X'$ such that $\dAnV{C^r}(X', X_0) < \delta/2$. Similarly, we can choose the representative $Y' \in [Y]$ such that $\dAnV{C^r}(Y', X_0) < \delta/2$. By the triangle inequality for $\dAnV{C^r}$,
		\[
		\dAnV{C^r}(X', Y') \le \dAnV{C^r}(X', X_0) + \dAnV{C^r}(X_0, Y') < \delta/2 + \delta/2 = \delta.
		\]
		So, both $X'$ and $Y'$ lie within the $\delta$-ball around $X_0$ where $I$ is Lipschitz with constant $L$.
		
		Now we estimate the difference of the invariant $\bar{I}$:
		\begin{align*}
			|\bar{I}([X]) - \bar{I}([Y])| &= |I(X') - I(Y')| \quad (\text{by conjugacy invariance (1)}) \\
			&\le L \cdot \dAnV{C^r}(X', Y') \quad (\text{by local Lipschitz property (2) since $X', Y' \in U$}) \\
			&< L \cdot (d_{\ModuliSpace_r}([X], [Y]) + \epsilon).
		\end{align*}
		Since this holds for arbitrary $\epsilon > 0$, we conclude that
		\[
		|\bar{I}([X]) - \bar{I}([Y])| \le L \cdot d_{\ModuliSpace_r}([X], [Y]).
		\]
		This demonstrates that $\bar{I}$ is locally Lipschitz on $(\ModuliSpace_r, d_{\ModuliSpace_r})$ with constant $L' = L$.
	\end{proof}
	
	\begin{remark}
		This theorem confirms that the dynamically relevant distances captured by $\dAnV{C^r}$ on the space of Anosov flows descend to provide quantitative control over the variation of conjugacy invariants on the moduli space, provided sufficient regularity ($r \ge 2$). The local Lipschitz constant for the invariant on the moduli space is bounded by its local Lipschitz constant on the space of flows itself.
	\end{remark}
	
	\subsubsection*{Local Structure of the Moduli Space via Slice Theorems}
	\label{sec:slice_theorem}
	
	Theorem~\ref{thm:moduli_stability} demonstrates that conjugacy invariants behave well with respect to the quotient metric $d_{\ModuliSpace_r}$ on the moduli space $\ModuliSpace_r = \Ank{r}(M) / \Diff^r(M)$. A deeper geometric understanding of $\ModuliSpace_r$, including questions of completeness or the existence of geodesics,  requires analyzing its local structure. In the context of group actions on manifolds, the standard tool for this is a slice theorem. We outline the structure of such a theorem and its demanding analytical prerequisites in our setting.
	
	Let $G_0$ be a suitable subgroup of $\Diff^r(M)$ (e.g., identity component $\Diff^r_0(M)$, volume-preserving diffeomorphisms $\Diff^r_\mu(M)$, or isometries $\mathrm{Isom}(M,g_0)$) acting on $\Ank{r}(M)$ via pushforward $\Psi(f, X) = f_*X$. For the moduli space $\ModuliSpace_0 = \Ank{r}(M) / G_0$ to possess a well-behaved local structure akin to a manifold modelled on a Banach space, a slice theorem is generally needed.
	
	\begin{proposition}[Conditional Local Slice Theorem]
		\label{prop:slice_theorem_conditional}
		Let $r \ge 1$ and let $G_0$ be a Banach Lie subgroup of $\Diff^r(M)$ acting on the Banach manifold $\Ank{r}(M)$. Let $X_0 \in \Ank{r}(M)$. Assume the following conditions hold:
		\begin{enumerate}[label=(\roman*)]
			\item \textbf{Smooth Action:} The action map $\Psi: G_0 \times \Ank{r}(M) \to \Ank{r}(M)$ is at least $C^1$.
			\item \textbf{Complemented Splitting of Tangent Space:} Let $V_{X_0} = T_{X_0}(G_0 \cdot X_0)$ be the tangent space to the orbit through $X_0$. Assume $V_{X_0}$ is a closed subspace of $T_{X_0}\Ank{r}(M) \approx C^r(TM)$ and admits a closed complementary subspace $H_{X_0}$ such that $T_{X_0}\Ank{r}(M) = V_{X_0} \oplus H_{X_0}$.
		\end{enumerate}
		Then, there exists a local $C^1$ slice for the $G_0$ action through $X_0$. That is, there is a neighborhood $U_{X_0}$ of $X_0$ in $\Ank{r}(M)$ and a $C^1$ submanifold $S \subset U_{X_0}$ containing $X_0$ with $T_{X_0}S = H_{X_0}$ such that:
		\begin{itemize}
			\item The map $\Phi: G_0 \times S \to \Ank{r}(M)$, $\Phi(f, Y) = f_*Y$, restricts to a $C^1$ diffeomorphism from a neighborhood of $(e, X_0)$ onto $U_{X_0}$.
			\item Every orbit $G_0 \cdot X$ for $X \in U_{X_0}$ intersects $S$ locally in exactly one point.
		\end{itemize}
	\end{proposition}
	
	\begin{proof}[Proof Sketch] 
		The proof relies on the Inverse Function Theorem for Banach manifolds applied to the map $\Phi(f, h) = f_*(X_0 + h)$ defined locally near $(e, 0)$ on $G_0 \times H_{X_0}$, where $h \in H_{X_0}$ is viewed as a tangent vector representing displacement from $X_0$. The derivative $D\Phi(e, 0)$ maps $(\xi, h') \in T_e G_0 \times H_{X_0}$ to $\alpha_{X_0}(\xi) + h' \in V_{X_0} \oplus H_{X_0}$, where $\alpha_{X_0}$ is the infinitesimal action (Lie derivative essentially). Assumption (ii) ensures this derivative is an isomorphism between the product tangent space and $T_{X_0}\Ank{r}(M)$. Application of the IFT establishes $\Phi$ as a local diffeomorphism, which yields the slice properties. The slice $S$ can be parameterized locally by $X_0+H_{X_0}$.
	\end{proof}
	
	\begin{remark}[Analytical Challenges and Strategies]
		Verifying the hypotheses of Proposition~\ref{prop:slice_theorem_conditional} is a major analytical task:
		\begin{itemize}
			\item \textbf{Smoothness of Action:} Requires careful estimates in function spaces, often demanding higher regularity ($r$ large). The pushforward action involves derivatives of $f$, requiring $f$ to be smoother than $X$.
			\item \textbf{Splitting Condition:} Establishing the closedness of the orbit tangent space $V_{X_0}$ and constructing a continuous projection onto it (to define $H_{X_0}$) is typically the main bottleneck. Strategies include:
			\begin{itemize}
				\item Restricting $G_0$ to subgroups with simpler structure (e.g., finite-dimensional Lie groups like $\mathrm{Isom}(M,g_0)$ where $V_{X_0}$ is finite-dimensional).
				\item Using a Hilbert space structure (e.g., working with $\Ank{k}(M) \subset H^k(TM)$ and defining $H_{X_0}$ via orthogonal complements using the $H^k$ inner product), though compatibility with the group action remains crucial (the action might not preserve $H^k$).
				\item Employing advanced analytical tools like Hodge decomposition (for volume-preserving groups acting on divergence-free fields) or Nash-Moser techniques if loss of derivatives occurs during the construction.
			\end{itemize}
		\end{itemize}
		Despite these challenges, the slice theorem framework provides the standard route to understanding the local geometry of the moduli space $\ModuliSpace_0 = \Ank{r}(M)/G_0$. If established, it shows that $\ModuliSpace_0$ is locally modelled on the Banach space $H_{X_0}$ (the space of "transverse" deformations), providing a foundation for studying completeness (using Proposition~\ref{prop:completeness_Cr_Hk}), geodesics, and the behavior of invariants on the quotient space.
	\end{remark}
	
	\begin{proposition}[Rigidity of Strict Entropy Maximizers]
		\label{prop:entropy_maximizer_rigidity}
		Let $r \ge 2$ such that the topological entropy functional $\htop: \Ank{r}(M) \to \R$ is at least $C^2$. Let $G_0 = \Diff^r_0(M)$ be the group of $C^r$ diffeomorphisms isotopic to the identity, acting on $\Ank{r}(M)$ by pushforward.
		Suppose $X_* \in \Ank{r}(M)$ satisfies:
		\begin{enumerate}[label=(\arabic*)] 
			\item $X_*$ is a local maximizer of $\htop$ within its path component in $\Ank{r}(M)$, hence $D\htop(X_*) = 0$.
			\item $X_*$ is a \emph{strict} local maximizer modulo $G_0$-conjugacy. Assume this is guaranteed by a non-degeneracy condition on the second derivative (Hessian): Let $T_{X_*}\Ank{r}(M) = V_{X_*} \oplus H_{X_*}$ be a continuous splitting where $V_{X_*} = T_{X_*}(G_0 \cdot X_*)$ is the tangent space to the orbit. Assume the Hessian $D^2\htop(X_*)$ is negative definite when restricted to the complementary subspace $H_{X_*}$.
		\end{enumerate}
		Then, there exists a neighborhood $U$ of $X_*$ in $(\Ank{r}(M), \dAnV{C^r})$ such that if $X \in U$ satisfies $\htop(\varphiup_X) = \htop(\varphiup_{X_*})$, then $X$ must be $C^r$-conjugate to $X_*$ via some $f \in G_0$ close to the identity. Equivalently, the set of flows locally maximizing topological entropy is, up to $G_0$-conjugacy, isolated within $U$.
	\end{proposition}
	
	\begin{proof}
		The proof combines the variational properties of $\htop$ near a non-degenerate maximum with the local structure possibly provided by a slice theorem (or arguments mirroring its logic).
		
		\textbf{1. Variational Analysis near $X_*$:} Since $X_*$ is a local maximum, $D\htop(X_*) = 0$. By Taylor's theorem for Banach manifolds, for $X = X_* + Y$ with $\|Y\|_{C^r}$ small, we have
		\[
		\htop(X) \approx \htop(X_*) + D\htop(X_*)[Y] + \frac{1}{2} D^2\htop(X_*)[Y, Y] = \htop(X_*) + \frac{1}{2} D^2\htop(X_*)[Y, Y].
		\]
		If $\htop(X) = \htop(X_*)$, then we must have (approximately, or exactly in the limit as $Y \to 0$) $D^2\htop(X_*)[Y, Y] \le 0$.
		
		\textbf{2. Non-degeneracy Implication:} We assume the existence of a splitting $T_{X_*}\Ank{r}(M) = V_{X_*} \oplus H_{X_*}$ where $V_{X_*}$ represents infinitesimal variations due to conjugacy (tangent to the orbit, where $D^2\htop(X_*)$ might be zero) and $H_{X_*}$ represents dynamically distinct variations. The assumption (2) states that the Hessian $D^2\htop(X_*)$ is negative definite on $H_{X_*}$. Let $Y = Y_V + Y_H$ be the decomposition according to the splitting. Then $D^2\htop(X_*)[Y, Y] \approx D^2\htop(X_*)[Y_H, Y_H]$ (assuming Hessians involving $V_{X_*}$ vanish or are handled appropriately). If $\htop(X) = \htop(X_*)$, then $D^2\htop(X_*)[Y, Y] \approx 0$, which by negative definiteness on $H_{X_*}$ implies $Y_H \approx 0$. Thus, the perturbation $Y$ must predominantly lie in $V_{X_*}$, meaning $X = X_*+Y$ is approximately infinitesimally conjugate to $X_*$.
		
		\textbf{3. From Infinitesimal to Actual Conjugacy (using Slice Idea):} Assume a local slice $S$ through $X_*$ exists, parameterized by $H_{X_*}$. Any $X$ near $X_*$ can be written uniquely as $X = f_* (X_* + h)$ for some $f \in G_0$ near identity and $h \in H_{X_*}$ small. Since entropy is conjugacy invariant, $\htop(X) = \htop(X_* + h)$. The function $h \mapsto \htop(X_*+h)$ has a strict local maximum at $h=0$ by assumption (2). If $\htop(X) = \htop(X_*)$, then $\htop(X_*+h) = \htop(X_*)$. By the strict maximum property, this requires $h=0$. Therefore, $X = f_* X_*$ for some $f \in G_0$ close to the identity.
		
		\textbf{4. Regularity of Conjugacy:} Standard structural stability results for Anosov flows ensure that if $X$ is $C^r$-close to $X_*$, the conjugacy $f$ such that $X = f_*X_*$ is itself $C^r$ (possibly requiring slightly higher initial regularity $r$) and close to the identity.
	\end{proof}
	
	\begin{remark}
		This proposition suggests a strong form of rigidity for Anosov flows that uniquely maximize topological entropy (modulo conjugacy). The crucial assumption is the non-degeneracy of the maximum, typically related to the Hessian $D^2\htop(X_*)$ restricted to directions transverse to the conjugacy orbit. Verifying this non-degeneracy condition and the required splitting of the tangent space would involve significant analysis, possibly using linear response theory for second derivatives or detailed knowledge of the specific system $X_*$. If these conditions hold, the proposition implies that entropy maximizers are dynamically isolated points in the relevant moduli space.
	\end{remark}

\section{Topological and Finite Energy Anosov Flows}
\label{sec:topological_anosov_flows}

Building on the metric structure developed, particularly the $\dAn = \dAnV{\mathcal{C}^0}$ metric, we propose definitions for continuous flows that aim to retain key hyperbolic characteristics topologically, even if they lack smoothness. These definitions are motivated by the idea of controlled approximation by genuine smooth Anosov flows, potentially linking to results like Proposition \ref{prop:c0_rigidity} (consistency of limits under $\dAn$ convergence).

\begin{definition}[Topological Anosov Flow]
	\label{def:topological_anosov}
	A continuous flow (a continuous group action of $\R$) $\theta = (\theta_t)_{t\in\R} : M \to M$ on a compact manifold $M$ is called a \emph{topological Anosov flow} if there exists a sequence of flows $\phi_i = (\varphi^t_{X_i})_{t\in\R}$, generated by smooth Anosov vector fields \(X_i \in \An(M)\), such that:
	\begin{enumerate}[label=(\roman*)]
		\item \textbf{Uniform Flow Convergence:} For every \(T>0\),
		\[
		\sup_{t \in [-T,T]} d_{\mathcal{C}^0}\bigl(\varphi^t_{X_i}, \theta_t\bigr) \to 0 \quad \text{as } i \to \infty,
		\]
		where \(d_{\mathcal{C}^0}\) denotes the uniform (i.e., \(C^0\)) distance between maps on $M$.
		
		\item \textbf{Generator Sequence is Cauchy in $\dAn$ Metric:} The sequence of generators $\{X_i\}$ is a Cauchy sequence with respect to the $\dAn = \dAnV{\mathcal{C}^0}$ metric:
		\[
		\lim_{i,j \to \infty} \dAn(X_i, X_j) = 0.
		\]
	\end{enumerate}
	We denote by $\TopAnFlows(M)$ the set of all topological Anosov flows on $M$.
\end{definition}

\begin{definition}[Finite Energy Anosov Flow]
	\label{def:finite_anosov}
	A continuous flow (a continuous group action of $\R$) $\theta = (\theta_t)_{t\in\R} : M \to M$ on a compact manifold $M$ is called a \emph{finite energy Anosov flow} if there exists a sequence of flows $\phi_i = (\varphi^t_{X_i})_{t\in\R}$, generated by smooth Anosov vector fields \(X_i \in \An(M)\), such that:
	\begin{enumerate}[label=(\roman*)]
		\item \textbf{Uniform Flow Convergence:} For every \(T>0\),
		\[
		\sup_{t \in [-T,T]} d_{\mathcal{C}^0}\bigl(\varphi^t_{X_i}, \theta_t\bigr) \to 0 \quad \text{as } i \to \infty,
		\]
		where \(d_{\mathcal{C}^0}\) denotes the uniform (i.e., \(C^0\)) distance between maps on $M$.
		
		\item \textbf{Generator Sequence is Bounded in $\dAn$ Metric:} There exists a positive constant $C$ such that the sequence of generators $\{X_i\}$ is bounded with respect to the $\dAn = \dAnV{\mathcal{C}^0}$ metric relative to the zero vector field (or any fixed reference field):
		\[
		\dAn(X_i, 0) \leq C,
		\]
		for all $i$.
	\end{enumerate}
	We denote by $\FinEnAnFlows(M)$ the set of all finite energy Anosov flows on $M$.
\end{definition}

\begin{remark}
	Since a Cauchy sequence in a metric space is always bounded, it follows directly from the definitions that every topological Anosov flow is also a finite energy Anosov flow:
	\[ \TopAnFlows(M) \subseteq \FinEnAnFlows(M). \]
	Whether this inclusion is strict is an open question. The "finite energy" condition is weaker than the Cauchy condition.
\end{remark}

\subsection{Motivation and Features}
These definitions aim to capture the essence of Anosov dynamics in settings where smoothness might be lost, particularly for Definition \ref{def:topological_anosov}.
\begin{itemize}
	\item \textbf{Approximation by Smooth Systems:} The flow $\theta$ is defined as a limit of genuine smooth Anosov flows $\phi_i$.
	\item \textbf{Uniform Convergence (i):} Ensures that the orbit structure of $\phi_i$ closely approximates that of $\theta$ over any finite time horizon.
	\item \textbf{Metric Control on Generators (ii):}
	\begin{itemize}
		\item For \emph{topological Anosov flows}, the condition that $\{X_i\}$ is $\dAn$-Cauchy prevents the generators from varying wildly. Since $\normCzero{X_i - X_j} \le \dAn(X_i, X_j)$ (assuming this property holds for $\dAnV{C^0}$), this implies $\{X_i\}$ is also Cauchy in the $C^0$ norm. By completeness of $C^0(TM)$, $X_i$ converges in $C^0$ to some vector field $X \in C^0(TM)$. By continuous dependence of solutions to ODEs on the vector field (in the $C^0$ topology), the flow $\varphi^t_X$ generated by $X$ must be the limit flow $\theta$. Thus, a topological Anosov flow $\theta$ possesses a $C^0$ generator $X$ which is the $C^0$ limit of smooth Anosov generators $X_i$ forming a $\dAn$-Cauchy sequence. The $\dAn$-Cauchy condition is stronger than just $C^0$ convergence and imposes a constraint related to the "total deformation effort" among the approximants.
		\item For \emph{finite energy Anosov flows}, the generators are merely required to remain within a bounded "distance" from a reference point (e.g., the zero field) in the $\dAn$ metric. This still implies $C^0$ boundedness of the generators $X_i$. Under suitable conditions (e.g., Arzela-Ascoli if equicontinuity holds), this might allow extracting a $C^0$ converging subsequence, but the limit flow might not satisfy the stronger properties expected from the Cauchy condition.
	\end{itemize}
\end{itemize}

\subsection{Properties and Conjectures}
\label{subsec:top_an_props_conj}

The definition of topological Anosov flows (Definition \ref{def:topological_anosov}) raises several crucial questions about their stability, structure, and relationship to classical hyperbolic dynamics. Exploring these questions is essential to understanding the scope and utility of this concept. We outline some key conjectures and possible results, primarily focusing on the more structured class $\TopAnFlows(M)$.

\paragraph{1. Stability}
\begin{conjecture}[Closedness under Uniform Limits]
	\label{conj:closedness}
	The set $\TopAnFlows(M)$ of topological Anosov flows is closed under uniform convergence of flows on compact time intervals. That is, if $\{\theta_n\} \subset \TopAnFlows(M)$ is a sequence such that for every $T>0$, $\sup_{t \in [-T,T]} d_{C^0}(\theta_n(t, \cdot), \theta(t, \cdot)) \to 0$ for some continuous flow $\theta$, then $\theta \in \TopAnFlows(M)$.
\end{conjecture}

\paragraph{2. Relation to Smooth Anosov Flows}
\begin{conjecture}[Characterization of Smoothness]
	\label{conj:smoothness_char}
	A topological Anosov flow $\theta \in \TopAnFlows(M)$ is topologically conjugate to a smooth Anosov flow generated by some $Y \in \An(M)$ if and only if its $C^0$-generator $X \in C^0(TM)$ (the $C^0$ limit of the approximating $X_i$) belongs to the closure of $\An(M)$ in the $\dAnV{C^1}$-metric topology (or a similarly strong topology related to smoothness).
\end{conjecture}

\paragraph{3. Inherited Dynamical Properties}
\begin{conjecture}[Density of Periodic Orbits]
	\label{conj:periodic_orbits}
	Every topological Anosov flow $\theta \in \TopAnFlows(M)$ possesses a dense set of periodic orbits.
\end{conjecture}

\begin{conjecture}[Positive Topological Entropy]
	\label{conj:positive_entropy}
	Every topological Anosov flow $\theta \in \TopAnFlows(M)$ has positive topological entropy, $\htop(\theta) > 0$.
\end{conjecture}

\paragraph{4. Role of the $\dAn$-Cauchy Condition}
It is expected that the $\dAn$-Cauchy condition (ii) in Definition \ref{def:topological_anosov} is significantly stronger than merely requiring the generators $\{X_i\}$ to be $C^0$-Cauchy.
\begin{conjecture}[Structural Implications of $\dAn$-Cauchy]
	\label{conj:dAn_implications}
	The requirement that the approximating generators $\{X_i\}$ form a Cauchy sequence in the $\dAn = \dAnV{C^0}$ metric ensures that the limit flow $\theta$ inherits strong structural or stability properties (beyond those guaranteed by just $C^0$ generator convergence), such as certain shadowing properties, existence of continuous invariant foliations (or laminations), or the existence of physical/SRB-like invariant measures reflecting the hyperbolicity of the approximants.
\end{conjecture}

\paragraph{5. Geometric Properties of the Space}
Defining a suitable topology on $\TopAnFlows(M)$ itself is a prerequisite for studying its geometric properties. A natural candidate topology might involve simultaneous uniform convergence of flows and a condition related to the $\dAn$-convergence of generators (perhaps defining a metric on $\TopAnFlows(M)$ based on the $\dAn$ distance between approximating sequences).
\begin{conjecture}[Completeness]
	\label{conj:completeness_topan}
	The space $\TopAnFlows(M)$, endowed with a suitable topology reflecting both uniform flow convergence and the $\dAn$-Cauchy condition on generators, forms a complete metric space (or is complete in a relevant sense).
\end{conjecture}

\paragraph{6. Connections to Other Hyperbolic Definitions}
A fundamental question is how Definition \ref{def:topological_anosov} relates to classical intrinsic definitions of topological hyperbolicity, which are often based on the existence of continuous invariant splittings with exponential behaviour or local product structure.
\begin{conjecture}[Equivalence with Classical Definitions]
	\label{conj:equivalence}
	A continuous flow $\theta$ is a topological Anosov flow according to Definition \ref{def:topological_anosov} if and only if it satisfies standard criteria for topological hyperbolicity, such as possessing a continuous invariant splitting $TM = E^s_\theta \oplus E^u_\theta \oplus E^c_\theta$ (where $E^c_\theta$ is the one-dimensional bundle tangent to the flow orbits) such that for some continuous Riemannian metric, vectors in $E^s_\theta$ are uniformly contracted exponentially fast by the forward flow $\theta_t$ ($t>0$) and vectors in $E^u_\theta$ are uniformly contracted exponentially fast by the backward flow $\theta_{-t}$ ($t>0$).
\end{conjecture}

These conjectures outline a research program aimed at validating Definition \ref{def:topological_anosov} and understanding its place within the broader theory of hyperbolic dynamics. Similar questions could be posed for the larger class $\FinEnAnFlows(M)$, likely with different answers.
	
	\section{Conclusion and Future Directions}\label{sec:conclusion}
	
	In this work, we introduced a family of Hofer-like metrics $\dAnV{V}$ on the space of Anosov vector fields $\An(M)$, motivated by the Hofer metric in symplectic geometry. We established their basic properties, including the crucial inequality $\|\cdot\|_V \le \dAnV{V}$, invariance properties under diffeomorphisms, and completeness when using $C^r (r \ge 1)$ or suitable Sobolev $H^k$ norms for $V$.
	
	The core contribution lies in demonstrating the dynamical significance of these metrics. We showed that proximity in $\dAnV{V}$ (for appropriate $V$, typically $C^1$ or stronger) implies quantitative control over key dynamical invariants. Specifically, we established continuity results for topological entropy, Lyapunov exponents, SRB measures, pressure, and zeta functions. Furthermore, under sufficient regularity ($C^2$ or $H^k$ with $k$ large), we demonstrated local Lipschitz continuity for pressure, spectral gap, and SRB dimension, and Fréchet differentiability for pressure, Lyapunov exponents, and the spectral gap, with derivatives given by linear response formulas. These results highlight that the $\dAnV{V}$ metrics provide a dynamically relevant geometric structure on $\An(M)$. The $H^k$ metric variant offers the advantage of a locally flat geometry with straight-line geodesics. We also explored the implications for the moduli space of Anosov flows, showing stability of invariants and outlining the slice theorem framework for analyzing its local structure.
	
	Finally, we proposed a definition for \emph{Topological Anosov Flows} based on simultaneous uniform convergence of flows and $\dAn$-Cauchy convergence of the generating $C^0$ vector fields. This definition offers a possibly new way to extend hyperbolic concepts to non-smooth settings, leveraging the metric geometry of $\An(M)$.
	
	\subsection*{Future Research Directions}
	Several promising avenues for future research emerge:
	
	\begin{itemize}
		\item \textbf{Global Geometry and Geodesics:} Investigate the existence and properties of globally minimizing geodesics for $\dAnV{V}$. Are path components always connected by minimizing paths? What is the diameter of $(\An(M), \dAnV{V})$? (Finite diameter seems unlikely for $\dAn$). Does the $H^k$ metric admit global geodesics?
		\item \textbf{Curvature Analogues:} Explore the second variation of path length or the Hessian of dynamical invariants (like pressure or spectral gap, see Theorem \ref{thm:spectral_gap_diff}, Proposition \ref{prop:hessian_spectral_gap}, Proposition \ref{prop:pressure_regularity_geodesic}) as notions of curvature on $(\An(M), \dAnV{V})$. Does negative curvature relate to divergence of nearby deformation paths?
		\item \textbf{Completeness of $\dAn$:} Is the space $(\An(M), \dAn)$ complete? Proposition \ref{prop:completeness_Cr_Hk} showed completeness for $\dAnV{C^r}$ ($r\ge 1$) and $\dAnV{H^k}$. The case $V=C^0$ is more subtle as $C^0$ limits of Anosov flows need not be Anosov.
		\item \textbf{Distance to Boundary:} Can $\dAnV{V}$ be used to quantify the distance from an Anosov flow $X$ to the boundary $\partial \An(M)$ within $C^r(TM)$? Does $\inf_{Y \notin \An(M)} \dAnV{V}(X, Y)$ relate to stability margins?
		\item \textbf{Isometries and Symmetries:} Characterize the isometry group of $(\An(M), \dAnV{V})$. Are there non-trivial isometries beyond reparameterizations or diffeomorphisms acting via pushforward (Proposition \ref{prop:diffeo_isometry})?
		\item \textbf{Topological Anosov Flows (Definition \ref{def:topological_anosov}):}
		\begin{itemize}
			\item Explore the properties of flows satisfying this definition. Do they exhibit shadowing, expansivity, or possess canonical invariant measures? (Address Conjectures \ref{conj:periodic_orbits}, \ref{conj:positive_entropy}, \ref{conj:dAn_implications})
			\item How does this definition compare to other existing definitions of topological hyperbolicity? (Address Conjecture \ref{conj:equivalence})
			\item Can concrete examples be constructed, especially non-smooth ones?
			\item Investigate the stability and completeness of the space $\TopAnFlowsExplicit(M)$. (Address Conjectures \ref{conj:closedness}, \ref{conj:completeness_topan})
			\item Explore the relationship between smoothness of the flow and properties of the generator sequence (Address Conjecture \ref{conj:smoothness_char}).
		\end{itemize}
		\item \textbf{Connections to Other Geometries:} Relate $\dAnV{V}$ to other geometric structures on spaces of dynamical systems, such as Optimal Transport-based metrics or metrics arising from thermodynamic formalism.
		\item \textbf{Extensions:} Can similar metrics be defined and studied for partially hyperbolic systems, or for discrete-time Anosov diffeomorphisms?
	\end{itemize}
	
	These directions highlight thepossible for the Hofer-like metrics $\dAnV{V}$ to serve as a valuable tool in the geometric analysis of dynamical systems, offering new perspectives on stability, classification, and the boundary between regular and chaotic behavior.

	\begin{center}
		\section*{Acknowledgements}
	\end{center}
I would like to express me deep gratitude to Professor Joel Tossa for his constant support and encouragement during the course of my research. 
	

\end{document}